%% file: main.tex
\documentclass[12pt]{article}
\usepackage[margin=1in]{geometry}
\usepackage{amsmath}
\usepackage{bbm}
\usepackage{url}
\usepackage[normalem]{ulem}
\usepackage{caption}
\usepackage{amssymb}
\usepackage{fourier}
\usepackage{framed}
\usepackage{color}
\usepackage{hyperref}
\definecolor{darkgreen}{rgb}{0.9, 0.9, 0.0}
\definecolor{darkred}{rgb}{0.55, 0.0, 0.0}
\definecolor{navyblue}{rgb}{0.0, 0.0, 0.7}
\hypersetup{
    colorlinks=true,
    urlcolor = {darkred},
    linkcolor = {darkred},
    citecolor = {darkred},
    linkcolor = {darkred},
    linktoc=all
}
\usepackage{amsthm}
\usepackage{graphicx}
\usepackage{tikz}
\usepackage{tikz-cd}
\usepackage{enumitem}
\usepackage{mathtools}
\usepackage{authblk}
\usepackage{extpfeil}
\usepackage{soul}

\usepackage{microtype}

\usetikzlibrary{matrix,arrows,decorations.pathmorphing,patterns,shapes.geometric}

\theoremstyle{definition}
\newtheorem{theorem}{Theorem}[section]
\newtheorem{lemma}[theorem]{Lemma}
\newtheorem{corollary}[theorem]{Corollary}
\newtheorem{proposition}[theorem]{Proposition}
\newtheorem{notation}[theorem]{Notation}
\newtheorem{example}[theorem]{Example}
\newtheorem{definition}[theorem]{Definition}
\newtheorem{convention}[theorem]{Convention}

\newtheorem{Remark}[theorem]{Remark}
\definecolor{darkblue}{rgb}{0.0, 0.0, 0.8}
\definecolor{darkred}{rgb}{0.8, 0.0, 0.0}
\definecolor{darkgreen}{rgb}{0.0, 0.5, 0.0}
\definecolor{lightgray}{rgb}{0.7, 0.7, 0.7}

\newcommand{\Full}{\mathrm{Full}}
\newcommand{\full}{\mathrm{full}}
\newcommand{\Hasse}{\mathrm{Hasse}}
\newcommand{\nbd}{\mathrm{nbd}}
\newcommand{\con}{\mathbf{Con}}
\newcommand{\rank}{\mathrm{rank}}
\newcommand{\inter}{\mathbf{Int}}

\newcommand{\A}{\mathcal{A}}
\newcommand{\I}{\mathcal{I}}
\newcommand{\id}{\mathrm{id}}
\newcommand{\ob}{\mathrm{ob}}
\newcommand{\EU}{\overline{\mathbf{U}}}
\newcommand{\ER}{\overline{\mathbf{R}}}
\newcommand{\Pb}{\mathbf{P}}
\newcommand{\Lb}{\mathbf{L}}

\newcommand{\Q}{\mathbf{Q}}
\newcommand{\vect}{\mathbf{vec}}
\newcommand{\sets}{\mathbf{set}}
\newcommand{\Vect}{\mathbf{Vec}}

\newcommand{\ab}{\mathbf{ab}}
\newcommand{\Ab}{\mathbf{Ab}}

\newcommand{\dom}{\mathrm{dom}}
\newcommand{\dero}{d_{\mathrm{E}}}

\newcommand{\Sets}{\mathbf{Set}}

\newcommand{\B}{\mathcal{B}}

\newcommand{\F}{\mathbb{F}}

\newcommand{\im}{\mathrm{im}}

\newcommand{\V}{\mathbb{V}}

\newcommand{\bott}{d_\mathrm{B}}

\newcommand{\ba}{\mathbf{a}}
\newcommand{\bb}{\mathbf{b}}

\newcommand{\bu}{\mathbf{u}}
\newcommand{\bv}{\mathbf{v}}

\newcommand{\Hrm}{\mathrm{H}}

\newcommand{\Z}{\mathbf{Z}}

\newcommand{\ZZ}{\mathbf{ZZ}}
\newcommand{\R}{\mathbf{R}}
\newcommand{\U}{\mathbf{U}}
\newcommand{\N}{\mathbf{N}}

\newcommand{\eps}{\varepsilon}
\newcommand{\dint}{d_{\mathrm{I}}}

\newcommand{\dgm}{\mathrm{dgm}}
\newcommand{\Dgm}{\mathrm{Dgm}}

\newcommand{\lmulti}{\left\{\!\!\left\{}
\newcommand{\rmulti}{\right\}\!\!\right\}}
\newcommand{\abs}[1]{\left\lvert#1\right\rvert}
\newcommand{\norm}[1]{\left\lVert#1\right\rVert}

\newcommand{\rk}{\mathrm{rk}}
\newcommand{\reeb}{\mathbf{Reeb}}

\newcommand{\dgmzz}{\mathrm{barc}^{\ZZ}}
\newcommand{\barc}{\mathrm{barc}}
\newcommand{\barcp}{\mathrm{barc}^{\Pb}}
\newcommand{\dgmpd}{\mathrm{dgm}^{\ZZ}}

\newcommand{\C}{\mathcal{C}}
\newcommand{\D}{\mathcal{D}}
\newcommand{\G}{\mathcal{G}}

\newcommand{\supp}{\mathrm{supp}}
\newcommand{\free}{L_{\F}}

\newcommand{\RNum}[1]{\uppercase\expandafter{\romannumeral #1\relax}}

\title{Generalized Persistence Diagrams for Persistence Modules over Posets}

\author[1]{Woojin Kim}
\author[2]{Facundo M\'emoli}

\affil[1]{Department of Mathematics, 
		Duke University.
		\thanks{\texttt{woojin@math.duke.edu}}}

\affil[2]{Department of Mathematics and Department of Computer Science and Engineering, 
		The Ohio State University.\thanks{\texttt{memoli@math.osu.edu}}}

\begin{document}
\maketitle

\begin{abstract}
    When a category $\C$ satisfies certain conditions, we define the notion of \emph{rank invariant} for arbitrary poset-indexed functors $F:\Pb\rightarrow \C$ from a category theory perspective. This generalizes the standard notion of rank invariant as well as Patel's recent extension. Specifically, 
    the barcode of any interval decomposable persistence modules $F:\Pb\rightarrow \vect$ of finite dimensional vector spaces can be extracted from the rank invariant by the principle of inclusion-exclusion. Generalizing this idea allows freedom of choosing the indexing poset $\Pb$ of $F:\Pb\rightarrow \C$ in defining Patel's generalized persistence diagram of $F$. Of particular importance is the fact that the generalized persistence diagram of $F$ is defined regardless of whether $F$ is  interval decomposable or not.
   
   By specializing our idea to zigzag persistence modules, we also show that the zeroth level set barcode of a Reeb graph can be obtained in a purely set-theoretic setting without passing to the category of vector spaces. This leads to a promotion of Patel's semicontinuity theorem about type $\mathcal{A}$ persistence diagram to Lipschitz continuity theorem for the category of sets.
\end{abstract}

\tableofcontents

\section{Introduction}
 The notion of persistence diagram  of  \cite{cohen2007stability} was recently extended  by Patel to the setting of constructible persistence modules $F:\R\rightarrow \C$ where $\C$ is an essentially small, symmetric monoidal category with images \cite{patel2018generalized}. In a nutshell, the persistence diagram of  $F:\R\rightarrow \C$ is obtained via M\"obius inversion of the map sending each pair $s\leq t$ in $\R$ to the \emph{image} of the morphism $F(s\leq t)$ in $\C$. This map can be regarded as a generalization of both the \emph{rank invariant} \cite{carlsson2009theory} and the \emph{persistent homology group} \cite{cohen2007stability}.  

 In this paper we further generalize the notions of  rank invariant and persistence diagram to the setting of locally finite poset-indexed functors $F:\Pb\rightarrow \C$, where, besides being symmetric, monoidal, essentially small and having images, the target category $\C$ is assumed to be bicomplete (see Convention \ref{convention} for another admissible set of assumptions about $\C$). A motivation for considering such level of generality is, as suggested by Patel, the possibility of studying torsion in data, e.g. persistent homology groups with integer coefficients. Specifically, we identify an explicit formula for the generalized persistence diagram of $F$ in terms of the generalized rank invariant of $F$ that we will define.  We in particular prove that, given an \emph{interval decomposable} diagram of vector spaces, its rank invariant and its persistence diagram can be translated into each other, and either of them contains enough information to reconstruct the diagram up to isomorphism.

One consequence of our framework is a novel method for computing the zeroth level set barcode of a Reeb graph (Section \ref{sec:reeb graph}).  This method is obtained by re-interpreting the barcode of a zigzag module \cite{carlsson2009zigzag} as the generalized persistence diagram. Furthermore, since the generalized persistence diagram is defined even for zigzag modules \emph{valued in categories other than the category of vector spaces}, our framework has significant potential to be exploited in the theoretical and algorithmic study on zigzag persistence  \cite{botnan2015interval,botnan2018algebraic,zigzag,carlsson2019parametrized,curry2016classification,elchesen2018reflection,milosavljevic2011zigzag,oudot2015zigzag}, and also in its applications to mobile sensor networks, image processing, analysis of time-varying metric spaces/graphs \cite{adams2015evasion,corcoran,kim2017stable,kim2018CCCG,mata2015zigzag}, etc.

\subsection{Our contributions}

By $\Pb$ denote any connected, locally finite poset (Definitions \ref{def:connected poset} and \ref{def:locally finite}), and by $\C$ denote any symmetric monoidal, essentially small, bicomplete category with images (see Convention \ref{convention} for another admissible set of assumptions about $\C$). 
\begin{enumerate}[label=(\roman*)]
\item We generalize the notion of rank invariant to functors $\Pb\rightarrow \C$ (i.e. generalized persistence modules \cite{bubenik2014categorification}).  In particular, our construction generalizes Patel's rank function \cite{patel2018generalized}: see Definition \ref{def:fully generalized rank invariant} and Example \ref{ex:fully general rank invariant} \ref{item:fully general rank invariant-1-parameter}. For any zigzag diagram of vector spaces, we show that our construction of the rank invariant is not only more faithful than the one introduced by V.~Puuska  \cite{puuska2017erosion}, but also becomes a \emph{complete} invariant: see Example \ref{ex:fully general rank invariant} \ref{item:fully general rank invariant-zigzag}, Theorem \ref{prop:completeness}, Remark \ref{rem:rk is complete}, and Appendix \ref{sec:puuska}. \label{item:rank invariant}
\end{enumerate}
In the sequel, we assume that the poset $\Pb$ is connected and \emph{essentially finite} -- a condition which is slightly stronger condition than being locally-finiteness; cf. Definition \ref{def:essentially}. Also, let $\vect$ be the category of finite-dimensional vector spaces over a fixed field $\F$.

\begin{enumerate}[resume,label=(\roman*)]

\item We extend the notion of generalized persistence diagram by Patel \cite{patel2018generalized} to arbitrary functors $\Pb\rightarrow \C$. In particular, for any interval decomposable persistence module  $F:\Pb\rightarrow \vect$,  the persistence diagram/barcode of $F$ can be extracted from the rank invariant by the \emph{principle of inclusion-exclusion} in combinatorics (Theorem \ref{prop:completeness}). This implies that our construction of rank invariant is \emph{complete} invariant for interval decomposable persistence modules. Also, as a by-product, we obtain a novel necessary condition for the interval decomposability of persistence modules $\Pb\rightarrow \vect$: See Remark \ref{rem:implication2} and Example \ref{ex:Amit}.\label{item:generalized PD}

\item It is well known that a Reeb graph can be seen as a zigzag persistence in the category of sets  \cite{curry2016classification,de2016categorified}. In this respect,  we show that the $0$-th level set barcode of a Reeb graph can be computed in a purely set-theoretic setting \emph{without passing to the category of vector spaces} (Section \ref{sec:reeb graph}). As a corollary, we partially promote the semicontinuity theorem by Patel to a Lipschitz continuity theorem (Corollary \ref{cor:promotion}) and this continuity theorem further extends to a certain class of Reeb graphs (Theorem \ref{thm:stability for sets}). These results indicate that the semicontinuity theorem by Patel is open to further improvement.\label{item:semicontinuity}\label{item:equivalence}
\end{enumerate}
\subsection{Key ideas: M\"obius inversion and a categorical view of the rank invariant}

This section aims at highlighting the main ideas of this paper, without discussing technical details. Specifically, we briefly overview how to generalize the notions of rank invariant and persistence diagram to the setting of finite poset-indexed diagrams;  strictly speaking, we do not actually require the indexing poset to be finite, but \emph{locally finite} (for defining the rank invariant) or \emph{essentially finite} (for defining the persistence diagram; Definition \ref{def:essentially}). 

Let $\Pb$ be a finite and connected poset (Definition \ref{def:connected poset}). Consider a diagram $F:\Pb\rightarrow \C$ where $\C$ is a symmetric monoidal, bicomplete category with images (all these terms are defined in Section \ref{sec:category}).  By combining the limit cone of $F$ and the colimit cocone of $F$, one obtains the canonical limit-to-colimit map $\phi_F:\varprojlim F \rightarrow \varinjlim F$.

Let $\con(\Pb)$ be the collection of all subposets $I$ of $\Pb$ such that the restriction of the Hasse diagram of $\Pb$ to $I$ is a connected graph (Definition \ref{def:interval poset}). For  $I\in \con(\Pb)$, we consider the restricted diagram $F|_{I}:I\rightarrow \C$ and define $\im(F|_{I})$ to be the \emph{image} of the canonical map $\phi_{F|_{I}}:\varprojlim F|_{I} \rightarrow \varinjlim F|_{I}$. Let $\I(\C)$ denote the collection of all isomorphism classes of $\C$. Now, as a generalization of the \emph{rank invariant} (a.k.a. \emph{rank function}) in \cite{carlsson2009theory,mccleary2018bottleneck,patel2018generalized}, we define the rank invariant of $F$ as the function
\[\rk(F):\con(\Pb)\rightarrow \I(\C),\] 
which sends $I\in\con(\Pb)$ to the isomorphism class of $\im(F|_I)$.
By the assumptions on $\C$, the codomain $\I(\C)$ is a symmetric monoid.\footnote{A symmetric monoid is a set $S$ equipped with a binary operation which is associative and symmetric. Also, $S$ must  contain an identity element.} Let $\A(\C)$ be the group completion of $\I(\C)$, known as the \emph{Grothendieck group} of $\C$. We define the persistence diagram of $F$ as the \emph{M\"obius inversion} of $\rk(F)$, which amounts to a function $\con(\Pb)\rightarrow \A(\C)$ (Definition \ref{def:generalized PD2}). These notions of rank invariant and persistence diagram generalize those appearing in \cite{botnan2018algebraic,carlsson2009zigzag,carlsson2009theory,cohen2007stability,patel2018generalized,ZC05}.

When restricted to the case when $\C=\vect$, related ideas have been or are currently being explored  by several researchers: Botnan, Oppermann and Steen have been studying the rank of $\phi_F$ for counting ``thin" summands in the indecomposable decomposition of $F:\Pb\rightarrow \Vect$. We remark that their work is addressing a version of Proposition \ref{rem:implication}, see announcement in \cite{botnan2017talk}. Also, Chambers and Letscher \cite{chambers2018persistent} define persistent homology for a filtration over a directed graph as the image of the canonical map between the limit and colimit. Their work also provides a variant of Proposition \ref{rem:implication} \cite[Lemma 3.1]{chambers2018persistent}.

Botnan and Lesnick exploited the notion of \emph{left Kan extensions} for addressing  stability of barcode of a zigzag module \cite{botnan2018algebraic}. We remark that $\rk(F)$ above can be interpreted as an invariant of $F$ which is obtained by interconnecting the \emph{left and right Kan extensions} of $F$ along a certain restriction functor. We refer the interested readers to the previous version (v4) of this paper; \cite[Sections E.2 and E.3]{kim2018generalized}.

\paragraph{Other related work.} In \cite{ogle2018structure}, Ogle provides a series of structure theorems about persistence modules over finite posets by taking into account inner product structures on persistence modules.

\paragraph{Organization.} In Section \ref{sec:prelim} we briefly review basic concepts from category theory and persistence theory,  and we also recall basic terminology from order theory. 
In Section \ref{sec:rank invariant in general} we define the rank invariant and persistence diagram of a functor  $\Pb\rightarrow \C$. In Section \ref{sec:reeb graph} we propose a novel viewpoint on Reeb graphs and their $0$-th level set barcodes. This new perspective  provides a new method to compute the $0$-th level set barcodes of Reeb graphs without passing to the category of vector spaces. Also, we establish stability results for the persistence diagrams of merge trees and ``untwisted" Reeb graphs. In Section \ref{sec:discussion} we discuss possible extensions.

\section{Preliminaries}\label{sec:prelim}
In Section \ref{sec:category} we set up terminology and notation relevant to category theory. In Section \ref{sec:details about interval modules} we review the notion of poset-indexed persistence modules and their interval decomposability. Also, we introduce notation relevant to zigzag persistence modules. In Section \ref{sec:hasse} we introduce the notion of path-connected subposets of a given poset which will be instrumental in the definitions of  generalized rank invariant and persistence diagram.

\subsection{Category theory elements.}\label{sec:category}

\paragraph{Categories.} Consult \cite{awodey2010category, mac2013categories} for general definitions related to category theory. For any category $\C$, let $\ob(\C)$ and $\hom(\C)$ denote the class/set of all objects and that of all morphisms in $\C$, respectively. Let $I$ be a small category, i.e. $\ob(I)$ and $\hom(I)$ are sets. For any two functors $F,G:I\rightarrow \C$, we write $F\cong G$ if $F$ and $G$ are naturally isomorphic. A functor $F:I\rightarrow \C$ will sometimes be referred to as a \emph{diagram}. Since the domain $I$ is small, we also refer to $F$ as a \emph{small} diagram. In particular, if  $\ob(I)$ and $\hom(I)$ are finite sets, then $F:I\rightarrow \C$ will be called a finite diagram. A sub-diagram of $F$ means the restriction $F|_J$ to a full subcategory $J$ of $I$.  The following categories will be of main interest in this paper.
\begin{enumerate}[label=(\roman*)]
    \item By $\Vect$ and $\vect$, we mean the category of vector spaces and \emph{finite} dimensional vector spaces, respectively with linear maps over a fixed field $\F$.
    \item By $\Sets$ and $\sets$, we mean the category of sets and \emph{finite} sets, respectively with set maps.
\end{enumerate}


Most of the following concepts can be found in \cite{mitchell1965theory,patel2018generalized,weibel2013k}.

\paragraph{Symmetric monoidal category with images.}\label{sec:terminology}
 A \emph{symmetric monoidal category} $(\C,\square)$ is, in brief, a category $\C$ with a binary operation $\square$ on its object and an identity object $e\in \ob(C)$ satisfying the following properties:

\begin{itemize}
    \item (Symmetry) $a\square b \cong b\square a$, for all $a,b\in \ob(C)$.
    \item (Associativity) $a\square(b\square c)\cong (a\square b)\square c$, for all $a,b,c\in \ob(C)$.
    \item (Identity) $a\square e \cong a$, for all $a\in \ob(C)$.
\end{itemize}
We refer the reader to \cite[p.114]{weibel2013k} for the precise definition of a symmetric monoidal
category. 

A morphism $f:a\rightarrow b$ is said to be a \emph{monomorphism} if $f$ is left-cancellative: for any morphisms $k_1, k_2: c \rightarrow a$, if $f\circ k_{1}=f\circ k_{2}$, then $k_{1}=k_{2}$. Such $f$ is often written as $f:a\hookrightarrow b$. On the other hand, a right-cancellative morphism $g:a\rightarrow b$ is said to be an \emph{epimorphism}, and often written as $g:a\twoheadrightarrow b$.

\begin{definition}[Images and epimorphic images]\label{def:images}
A morphism $f:a\rightarrow b$ has an \emph{image} if there exist a monomorphism $h:\im(f)\hookrightarrow b$ and a morphism $g:a\rightarrow \im(f)$ such that $f=h\circ g$, satisfying the following property: For any monomorphism $h':z\hookrightarrow b$ and a morphism $g':a\rightarrow z$ with $f=h'\circ g'$, there is a unique morphism $u:\im(f)\rightarrow z$ such that the following diagram commutes:

\[\begin{tikzcd}a\arrow[->]{rr}{f} \arrow{rd}{g} \arrow{rdd}[swap]{g'}&&b\\
&\im(f) \arrow[hook]{ru}{h}\arrow[dashed]{d}{u}\\
& z \arrow[hook]{ruu}[swap]{h'}
\end{tikzcd}
\]
If, moreover, the morphism $g$ is an epimorphism, then $f$ is said to have an \emph{epimorphic image}.
\end{definition}

\begin{Remark}\label{rem:monic}The morphism $u$ in the diagram above is a monomorphism: Assume that there exist two morphisms $k_1,k_2:c\rightarrow \im(f)$ with $u\circ k_1=u\circ k_2$. Then,
\begin{align*}
    u\circ k_1 =u\circ k_2  &\implies h'\circ (u\circ k_1) = h'\circ (u\circ k_2)\\
    &\implies (h'\circ u)\circ k_1 = (h'\circ u)\circ k_2\\
    &\implies h\circ k_1 = h\circ k_2\\
    &\implies  k_1 = k_2&\mbox{since $h$ is a monomorphism.}
\end{align*}
\end{Remark}

A category in which every morphism has an image (resp. epimorphic image) is said to be a \emph{category with images (resp. category with epimorphic images)}. See \cite[p.12]{mitchell1965theory} for more about images.

\paragraph{Complete, cocomplete, essentially small categories} \cite{mac2013categories}.
 A category $\C$ is called \emph{complete} if every small diagram $F:I\rightarrow \C$ has a limit in $\C$. Likewise, a category $\C$ is called \emph{cocomplete} if every small diagram has a colimit in $\C$. If a category $\C$ is both complete and cocomplete, then $\C$ is called \emph{bicomplete}.\footnote{Besides $\Sets$ and  $\Vect$, examples include the category of groups, the category of abelian groups, the category of topological spaces.} A category $\C$ is called \emph{essentially small} if the collection of isomorphism classes of objects in $\C$ is a set.

\begin{convention}\label{convention} Throughout this paper, unless otherwise stated, we always assume that $(\C,\square)$ (or simply $\C$) is essentially small and symmetric monoidal.
Also, we always assume that at least one of conditions (a) and (b) below hold:
\begin{enumerate}[label=(\alph*)]
    \item $\C$ is bicomplete and has images.\label{item:option1}
    \item $\C$ is a full subcategory\footnote{A subcategory $\C'$ of a category $\C$ is \emph{full} if for all $a,b\in \ob(\C')$, $\hom_{\C'}(a,b)=\hom_{\C}(a,b)$.} of a bicomplete category $\D$, where \label{item:a full subcategory of a bicomplete category}
   \begin{itemize}
        \item $\D$ has \emph{epimorphic} images 
        \item  for any monomorphism $f:a\hookrightarrow b$ in $\D$, if $b\in \ob(\C)$, then $a\in \ob(\C)$.\label{item:mono}
        \item  for any epimorphism  $g:c\twoheadrightarrow d$ in $\D$, if $c\in \ob(\C)$, then $d\in \ob(\C)$.\label{item:epi}
    \end{itemize}
\end{enumerate}
\end{convention}

For example, the categories  $(\sets,\sqcup)$ (where $\sqcup$ stands for disjoint union), $(\vect,\oplus)$ (where $\oplus$ stands for direct sum), and $(\ab,\oplus)$ (the category of finitely generated abelian groups) all conform with Convention \ref{convention}. In particular, they are full subcategories of the bicomplete categories $\Vect$, $\Sets$, and $\Ab$ (the category of abelian groups) respectively whose images are epimorphic.

\paragraph{Grothendieck groups.} We review the notion of Grothendieck groups from  \cite[Chapter II]{weibel2013k} and \cite[Section 6]{patel2018generalized}. Fix a category $(\C,\square)$. Let $\I(\C)$ be the set of all isomorphism classes of $\C$. For any $C\in \ob(\C)$, let $[C]$ denote the isomorphism class of $C$ in $\I(\C).$ The binary operation in $\I(\C)$ is defined as $[a]+[b]:=[a\square b]$. The group completion $\A(\C)$ of $\I(\C)$ with respect to $+$ is said to be the \emph{Grothendieck group} of $\C$ . Hence, the notions of addition $+$ and subtraction $-$ are well-defined in $\A(\C)$. When $\C$ is an abelian category, there is another type of Grothendieck group $\B(\C)$.\footnote{Define the relation $\sim$ on $\A(\C)$ as $[b]\sim [a]+[c]$ if there exists a short exact sequence $0\rightarrow a\rightarrow b\rightarrow c\rightarrow 0$. Then, $\B(C)$ is defined as the quotient group $\A(\C)/\sim$.   The category $\vect$ is an instance of abelian categories, and it holds that $\A(\vect)=\B(\vect)$ \cite[Example 6.2.1]{patel2018generalized}.} However, it is beyond the scope of this paper to give a complete treatment of $\B(\C)$.

\begin{Remark}
\label{rem:counterpart} For the monoid $(\I(\vect),\bigoplus)$ (resp. $(\I(\sets),\bigsqcup)$) of isomorphism classes, we have the isomorphism 
\[(\I(\vect),\bigoplus)\cong (\Z_+,+)\ \mbox{(resp.  $(\I(\sets),\bigsqcup)\cong (\Z_+,+))$},
\] which sends each isomorphism class $[V]\in\I(\vect)$ to $\dim(V)$ (resp. $[A]\in\I(\sets)$ to $\abs{A}$). Extending this isomorphism, we obtain the group isomorphism $(\A(\vect),\bigoplus)$ $\cong(\Z,+)$ (resp. $(\A(\sets),\bigsqcup)\cong(\Z,+)$.). 
\end{Remark}

\subsection{Interval decomposability and barcodes}\label{sec:details about interval modules}

Given any poset $\Pb$, we regard $\Pb$ as the category: Objects are elements of $\Pb$. For any $p,q\in \Pb$, there exists a unique morphism $p\rightarrow q$  if  and only if $p\leq q$.

For $\Pb$ a poset and $\C$ an arbitrary category, $F:\Pb\rightarrow \C$ a functor, and $s\in \Pb$, let $F_s:=F(s)$. Also, for any pair $s\leq t$ in $\Pb$, let $\varphi_F(s,t):F_s\rightarrow F_t$ denote the morphism $F(s\leq t).$ 
Any functor $F:\Pb\rightarrow \vect$ is said to be a \emph{$\Pb$-indexed module}.

\paragraph{Interval modules and direct sums.}  We mostly follow the notation and definitions from \cite{botnan2018algebraic}.

\begin{definition}[Connected posets]\label{def:connected poset} A poset $\Pb$ is said to be \emph{connected}, if for all $p,q\in \Pb$, there exists a sequence $p=p_0,\ldots,p_n=q$ in $\Pb$ such that $p_i$ and $p_{i+1}$ are comparable, i.e. $p_i\leq p_{i+1}$ or $p_{i+1}\leq p_i$,  for $i=0,\ldots,n-1$.
\end{definition}  
 For a poset $\Pb$, non-empty \emph{convex} connected subposets of $\Pb$ are said to be \emph{intervals} of $\Pb$:
\begin{definition}[Intervals]\label{def:intervals} Given a poset $\Pb$, an \emph{interval} $\mathcal{J}$ of $\Pb$ is any non-empty subset $\mathcal{J}\subset \Pb$ such that (1) $\mathcal{J}$ is connected, and (2) if $r,t\in \mathcal{J}$ and $r\leq s\leq t$, then $s\in \mathcal{J}$.

By $\inter(\Pb)$, we denote the collection of all intervals of $\Pb$.	
\end{definition}

\label{interval module}For $\mathcal{J}$ an interval of $\Pb$, the \emph{interval module} ${I}^{\mathcal{J}}:\Pb\rightarrow \vect$ is the $\Pb$-indexed module where 

\[{I}_t^{\mathcal{J}}=\begin{cases}
\mathbb{F}&\mbox{if}\ t\in \mathcal{J},\\0
&\mbox{otherwise.} 
\end{cases}\hspace{20mm} \varphi_{I^{\mathcal{J}}}(s,t)=\begin{cases} \mathrm{id}_\mathbb{F}& \mbox{if} \,\,s,t\in\mathcal{J},\ s\leq t,\\ 0&\mbox{otherwise.}\end{cases}\]

Let $F,G$ be $\Pb$-indexed modules. The \emph{direct sum $F\bigoplus G$ of $F$ and $G$} is the $\Pb$-indexed module defined as follows: for all $s\in \Pb$, $(F\bigoplus G)_s:=F_s\bigoplus G_s$ and for all $s\leq t$ in $\Pb$, the linear map $\varphi_{F\bigoplus G}(s,t):(F\bigoplus G)_s \rightarrow (F\bigoplus G)_t$ is defined by  \[\varphi_{F\bigoplus G}(s,t)(v,w):=\big(\varphi_F(s,t)(v),\varphi_G(s,t)(w)\big)\] for all $(v,w)\in (F\bigoplus G)_s$. We say that a $\Pb$-indexed module $F$ is \textit{decomposable} if $F$ is (naturally) isomorphic to $G_1\bigoplus G_2$ for some non-trivial $\Pb$-indexed modules $G_1$ and $G_2$, and we denote it by  $F\cong G_1\bigoplus G_2$. Otherwise, we say that $F$ is \textit{indecomposable}. 

\begin{proposition}[{\cite[Proposition 2.2]{botnan2018algebraic}}] ${I}^\mathcal{J}$ is indecomposable.
\end{proposition}

\paragraph{Barcodes.} Recall that \textit{a multiset} is a collection of objects (called elements) in which elements may occur more than once. We call the number of instances of an element in a specific multiset \textit{the multiplicity} of the element. For example, $A=\lmulti x,x,y\rmulti$ is a multiset and the multiplicity of $x$ is two. Also, this multiset $A$ is distinct from the multiset $\lmulti x,y \rmulti$.

\begin{definition}[Interval decomposability]\label{def:interval decomposable}A $\Pb$-indexed module $F$ is \emph{interval decomposable} if there exists a multiset $\barcp(F)$ of intervals (Definition \ref{def:intervals}) of $\Pb$ such that 
\[F\cong \bigoplus_{\mathcal{J}\in \barcp(F)}I^{\mathcal{J}}\ \ \ \]
\end{definition}

It is well-known that, by the theorem of Azumaya-Krull-Remak-Schmidt \cite{azumaya1950corrections}, such a decomposition is unique up to a permutation of the terms in the direct sum. Therefore, the multiset $\barcp(F)$ is unique if $F$ is interval decomposable  since a multiset is careless of the order of its elements. 
\begin{definition}[Barcodes]\label{def:barcode}
We call $\barcp(F)$ in Definition \ref{def:interval decomposable} the \emph{barcode} of $F.$
\end{definition}


Given any two posets $\Pb,\Pb'$, we assume that by default the product $\Pb\times\Pb'$ is equipped with the partial order where $(p,p')\leq (q,q')$ if and only if $p\leq p'$ and $q\leq q'$. 
Also, by $\Pb^\mathrm{op}$ we mean the opposite poset of $\Pb$, i.e. for any $p,q\in \Pb$,  $p\leq q$ in $\Pb^{\mathrm{op}}$ if and only if $q\leq p$ in $\Pb$.

\begin{figure}
    \centering
      \includegraphics[width=0.35\textwidth]{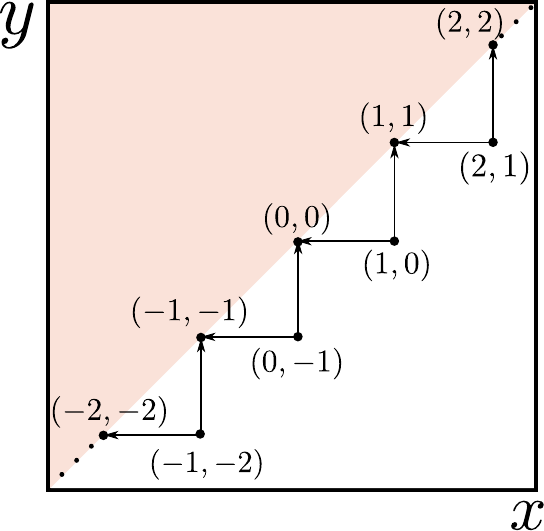}
     \caption{The posets $\ZZ$ in Definition \ref{def:posets}.}
     \label{fig:ZZ}
\end{figure}

\begin{definition}[Zigzag poset]\label{def:posets} \label{item:ZZ} The poset $\ZZ=\{(i,j)\in \Z^2: j=i\ \mbox{or}\ j=i-1\}$ has the partial order inherited from $\R^{\mathrm{op}}\times \R$. See Figure \ref{fig:ZZ}.
\end{definition}

\begin{definition}[Zigzag module]Any functor $F:\ZZ\rightarrow \vect$ is called a \emph{zigzag module}.
\end{definition}

\begin{theorem}[Interval decomposability of 1-D persistence modules and zigzag modules {\cite{botnan2015interval,zigzag,crawley2015decomposition}}]\label{thm:interval decomposition1}\label{thm:interval decomposition0}  For any $\Pb \in \{\R,\Z,\ZZ\}$, $\Pb$-indexed modules are interval decomposable. Also, for each interval $I$ of $\ZZ$ (Definition \ref{def:intervals}), $I$-indexed modules are interval decomposable.
\end{theorem}

\paragraph{The barcode associated to a zigzag module.} By Theorem \ref{thm:interval decomposition0}, any zigzag module $M:\ZZ\rightarrow \vect$ admits a \emph{barcode} (Definition \ref{def:barcode}), which will be denoted by $\dgmzz(M)$.

\begin{notation}[Intervals of $\ZZ$]\label{nota:intervals of zz} The notation introduced in \cite{botnan2018algebraic} is useful for describing barcodes of zigzag modules: Letting $\preceq$ (resp. $\prec$) denote the partial order (resp. the strict partial order) on $\Z^2$ (not on $\Z^{\mathrm{op}}\times \Z$), any interval of $\ZZ$ falls into  one of the following four types: 
	\begin{align*}
	(b,d)_{\ZZ}&:=\{(i,j)\in \ZZ: (b,b)\prec(i,j)\prec(d,d)\}&&\mbox{for\   $b<d$ in $\Z\cup \{-\infty,\infty\}$},\\
	[b,d)_{\ZZ}&:=\{(i,j)\in \ZZ: (b,b)\preceq (i,j)\prec(d,d)\}&&\mbox{for    $b<d$ in  $\Z\cup \{\infty\}$},\\
	(b,d]_{\ZZ}&:=\{(i,j)\in \ZZ: (b,b)\prec (i,j)\preceq(d,d)\}&& \mbox{for    $b<d$ in  $\Z\cup \{-\infty\}$},\\
	[b,d]_{\ZZ}&:=\{(i,j)\in \ZZ: (b,b)\preceq (i,j)\preceq(d,d)\}&& \mbox{for   $b\leq d$ in $\Z$}.
	\end{align*}
	
See Figure \ref{fig:intervals} for examples. Specifically, we let $\langle b,d \rangle_{\ZZ}$ denote any of the above types of intervals. By utilizing this notation, the barcode of a zigzag module $M:\ZZ\rightarrow \vect$ can be expressed, for some index set $J$, as 
\[\dgmzz(M)=\lmulti\langle b_j,d_j\rangle_\ZZ: M\cong \bigoplus_{j\in J}I^{\langle b_j,d_j \rangle_\ZZ} \rmulti.\] Here, $\lmulti\cdot\rmulti$ is used instead of $\{\cdot\}$ to indicate $\dgmzz(M)$ is a \emph{multiset.}

\end{notation}	
\begin{figure}
\includegraphics[width=\textwidth]{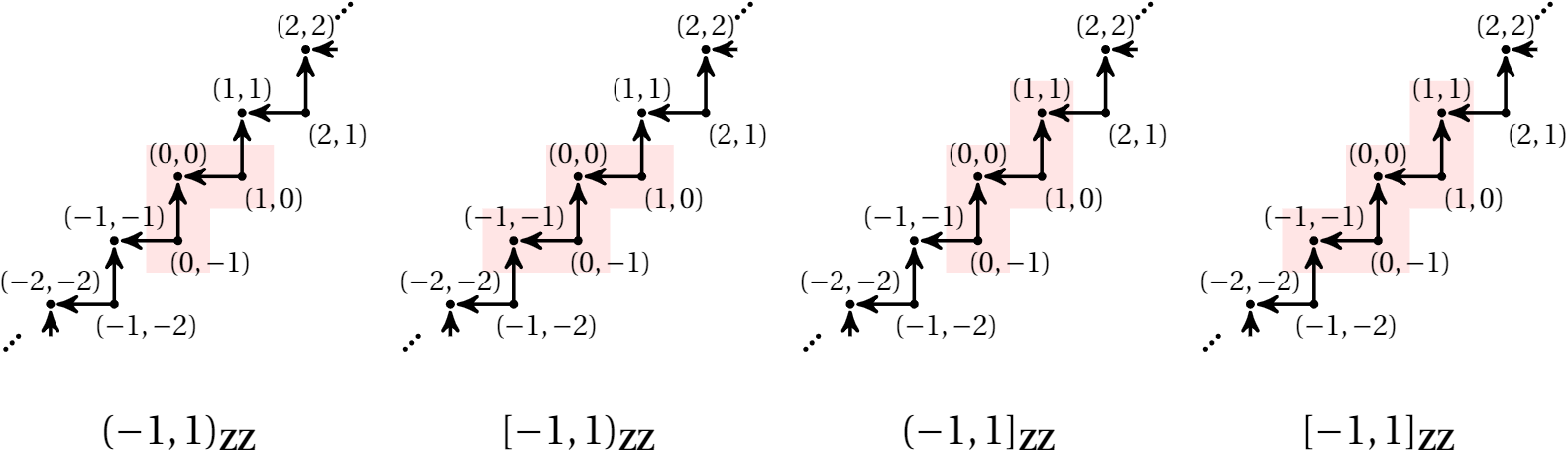}
	\caption{\label{fig:intervals}The points falling into the shaded regions comprise the intervals $(-1,1)_\ZZ, [-1,1)_\ZZ, (-1,1]_\ZZ$ and $[-1,1]_\ZZ$ of the poset $\ZZ$, respectively in order.}
\end{figure}

\subsection{Hasse diagram and path-connected subposets of a poset.}\label{sec:hasse}

We introduce  the notion of \emph{path-connected subposets} of a poset (this concept is \emph{different} from that of \emph{connected posets} in Definition \ref{def:connected poset}). To that end, we begin by reviewing relevant terminology from lattice theory \cite{birkhoff1940lattice}.

\begin{definition}[Locally finite posets]\label{def:locally finite} A poset $\Pb$ is said to be \emph{locally finite} if for all $p,q\in \Pb$ with $p\leq q$, the set $[p,q]:=\{r\in \Pb: p\leq r\leq q\}$ is finite.
\end{definition}

Let $\Pb$ be a poset and let $p,q\in \Pb$. We say that $p$ \emph{covers} $q$ if $q< p$ and there is no $r\in \Pb$ such that $q<r<p$. If either [$p$ covers $q$] or [$q$ covers $p$], then we write $p\asymp q$.

\begin{definition}[Hasse diagram of a poset]   The \emph{Hasse diagram} of a poset $\Pb$ is a simple graph on the vertex set $\Pb$, denoted by $\Hasse(\Pb)$, where any two different vertices $p,q\in \Pb$ are adjacent if and only if $p \asymp q$.
\end{definition}
 
  If $\Pb$ is connected (Definition \ref{def:connected poset}) and locally finite, then it is clear that $\Hasse(\Pb)$ is a connected graph. Locally finiteness is an important assumption for studying the connectedness of the Hasse diagram of a poset. For example, even though $\R$ is a connected totally ordered set, the Hasse diagram of $\R$ contains no edge.

\begin{definition}[Path-connected subposets]\label{def:interval poset}  Let $\Pb$ be a poset. A  subposet $I\subset \Pb$  is said to be \emph{path-connected} in $\Pb$ if the induced subgraph of $\Hasse(\Pb)$ on $I$ is a connected graph. 

By $\con(\Pb)$ we denote the collection of all path-connected subposets of $\Pb$. We equip $\con(\Pb)$ with the partial order given by inclusion, i.e. a pair $I,J\in \con(\Pb)$ satisfies $I\leq J$ if and only if  $I\subset J$.
\end{definition}

We emphasize the difference between  connectivity and path-connectivity  (Definitions \ref{def:connected poset} and \ref{def:interval poset}) as follows:  If a subposet $\Q$ is path-connected in $\Pb$, then $\Q$ itself is connected. However, in general, the converse is not true. For example, given the poset $\{a\leq b\leq c\}$, the subposet $\{a\leq c\}$ is connected itself, but \emph{not} path-connected in $\{a\leq b\leq c\}$. 

\begin{Remark}\label{rem:con(p)}
    Let $\Pb$ be a locally finite poset. Note that the collection $\inter(\Pb)$ of all intervals of $\Pb$ (Definition \ref{def:intervals}) is contained in $\con(\Pb)$. Also, $\inter(\Pb)$ can sometimes be equal to $\con(\Pb)$, e.g.  $\Pb=\Z$ or $\Pb=\ZZ$.\label{item:con(p)2}
\end{Remark}


\section{Generalized rank invariant and generalized persistence diagrams}\label{sec:rank invariant in general}

In this section we introduce the notions of \emph{generalized rank invariant} and \emph{generalized persistence diagram}. 

\subsection{Rank invariant for persistence modules over posets}\label{sec:the rank invariant}

In this section, we introduce the notion of \emph{generalized rank invariant} of a diagram $\Pb\rightarrow \C$ when $\Pb$ is a locally finite, connected poset (as for assumptions on the target category $\C$, recall Convention \ref{convention}).

\paragraph{Rank and rank invariant of a persistence module over a poset.} Recall the notions of cone and cocone from category theory (Definitions \ref{def:cone} and \ref{def:cocone}). The following observation is the first step toward defining the rank invariant of a persistence module over a poset.

\begin{proposition}[Canonical map from a cone to a cocone]\label{prop:canonical}
Let $\Pb$ be a connected poset and let $F:\Pb\rightarrow\C$. Let $\left(L,(\pi_a)_{a\in \Pb}\right)$ and $\left(C,(i_a)_{a\in \Pb}\right)$ be any cone and cocone of $F$, respectively. Then for any $a,b\in \Pb$, we have that $i_a\circ \pi_a= i_b\circ \pi_b.$
\end{proposition}

\begin{proof} Fix $a,b\in \Pb$ and let $a=c_1,c_2,\ldots,c_n=b$ be any sequence in $\Pb$ such that $c_j$ and $c_{j+1}$ are comparable for $j=1,\ldots,n-1$. Let $V_j:=M_{c_j}$ for each $j$. Then, we have the combined commutative diagram of $\left(L,(\pi_{c_j})_{j=1,\ldots,n}\right)$ and $\left(C,(i_{c_j})_{j=1,\ldots,n}\right)$ as follows:
\[\begin{tikzcd}
		&&C&&\\ \\V_1\arrow[<->]{r}[description]{f_1}\arrow{rruu}[description]{i_{c_1}}&V_2\arrow{ruu}[description]{i_{c_2}}&V_3\arrow[<->]{l}[description]{f_2}\arrow[<->]{r}[description]{f_3}\arrow{uu}[description]{i_{c_3}}&\cdots\arrow[<->]{r}[description]{f_{n-1}}\arrow{luu}&V_n\arrow{lluu}[description]{i_{c_n}}\\ \\ 
		&&L\arrow{lluu}[description]{\pi_{c_1}}\arrow{luu}[description]{\pi_{c_2}}\arrow{uu}[description]{\pi_{c_3}}\arrow{ruu}\arrow{rruu}[description]{\pi_{c_n}}&&
		\end{tikzcd}
\]
Without loss of generality, assuming $f_1$ is a map from $V_1$ to $V_2$, we prove that $i_{c_1}\circ \pi_{c_1}= i_{c_2}\circ \pi_{c_2}$. This is clear by tracking the commutativity of the diagram above: \[i_{c_1}\circ \pi_{c_1} = (i_{c_2}\circ f_1)\circ \pi_{c_1} = i_{c_2}\circ (f_1\circ \pi_{c_1}) = i_{c_2}\circ \pi_{c_2}.\]
Similarly, for each $k=2,\ldots, n-1$, we have $i_{c_k}\circ\pi_{c_k}=i_{c_{k+1}}\circ \pi_{c_{k+1}},$ completing the proof.
\end{proof}

Recall that a limit (resp. colimit) of a functor $F$ is the terminal (resp. initial) object in the category of cones (resp. cocones) over $F$ (Definitions \ref{def:limit} and \ref{def:colimit}). By virtue of Proposition \ref{prop:canonical}, we can define:

\begin{definition}[Rank of a poset-indexed diagram]\label{def:fully general rank} Let $\Pb$ be a connected poset.  The \emph{rank} of an $F:\Pb\rightarrow \C$ is defined as the isomorphism class of the image of the canonical limit-to-colimit map $\psi_F:\varprojlim F\rightarrow \varinjlim F$.
\end{definition}

In what follows, for any $F:\Pb\rightarrow \vect$, we will regard the codomain of $\rk(F)$ as $\Z_+$ by identifying $m\in \Z_+$ with $[\F^m]\in \I(\vect)$ (Remark \ref{rem:counterpart}). 

\begin{Remark}\label{rem:rank is finite}
In Definition \ref{def:fully general rank}, assuming $\C=\vect$, the rank of $F$ cannot exceed $\min_{\ba\in \Pb}\dim(F_\ba)$ and is thus finite. More generally,  $\im(\psi_F)$ belongs to $\I(\C)$ for any category $\C$ satisfying Convention \ref{convention}.
\end{Remark} 
The example below justifies the use of the term \emph{rank}.

\begin{example}\label{ex:trivial cases} 
\begin{enumerate}[label=(\roman*)]
    \item  For the singleton set $\{\bullet\}$, a functor $F:\{\bullet\}\rightarrow \vect$  amounts to a vector space $F_{\bullet}$ with the identity map $\id_{F_{\bullet}}:F_{\bullet}\rightarrow F_{\bullet}$. In this case, the rank of $F$ is the dimension of $F_{\bullet}$. \label{item:trivial cases1}
    \item A functor $G:\{a\leq b\}\rightarrow \vect$  amounts to the linear map $G_{a}\stackrel{g}{\longrightarrow}G_{b}$, where $g=\varphi_G(a\leq b)$. Note that $\varprojlim G \cong G_{a}$ and $\varinjlim G \cong G_{b}$ and the rank of $G$ is identical to that of $g$. \label{item:trivial cases2}
\end{enumerate}
\end{example}

\begin{definition}[Generalized rank invariant]\label{def:fully generalized rank invariant} Let $\Pb$ be a locally finite, connected poset.  We define the \emph{(generalized) rank invariant} of a diagram $F:\Pb\rightarrow \C$ as the map \[\rk(F):\con(\Pb)\rightarrow \I(\C)\]
which sends each $I\in \con(\Pb)$ to the rank of the sub-diagram $F|_{I}$. 
\end{definition}

In the following example, we will see that the generalized rank invariant is either \emph{equivalent to} or \emph{a refinement of} the rank invariant that have been considered in the literature \cite{carlsson2009theory,cerri2013betti,patel2018generalized,puuska2017erosion}:
\begin{example}\label{ex:fully general rank invariant} 
\begin{enumerate}[label=(\roman*)]
    \item (1-parameter persistence) Consider the integers $\Z$. We have $\con(\Z)=\inter(\Z)$ and the generalized rank invariant of a diagram $F:\Z\rightarrow \vect$ coincides with the standard rank invariant \cite{carlsson2009theory,cohen2007stability}. This directly follows from the following observation: for any $a,b\in \Z$ with $a\leq b$, the rank of $\varphi_F(a, b)$ coincides with the rank of $F|_{[a,b]}$ given in Definition \ref{def:fully general rank}. This observation is still valid even if the target category of $F$ is other than $\vect$, implying that the rank invariant of $\Z\rightarrow \C$ in Definition \ref{def:fully generalized rank invariant} is essentially equal to that of \cite{patel2018generalized}. \label{item:fully general rank invariant-1-parameter}

    \item (Zigzag persistence)  Consider the zigzag poset $\ZZ$. Note that $\con(\ZZ)=\inter(\ZZ)$. The generalized rank invariant of a diagram $F:\ZZ\rightarrow \vect$ contains more information than the rank invariant defined in \cite{puuska2017erosion}. See Appendix \ref{sec:puuska}. More interestingly,  even though $\rk(F)$ is defined in way which makes no reference to whether  $F$ is interval decomposable or not , it will turn out that one can extract the barcode of $F$ from $\rk(F)$ (see Figure \ref{fig:inclusion-exclusion} (A)). This implies that the generalized rank invariant is a \emph{complete} invariant for zigzag modules, i.e. any two zigzag modules that have the same rank invariant are isomorphic. This is a generalization of the completeness of the rank invariant for $\Z$-indexed persistence (Theorem \ref{thm:persistence module}). \label{item:fully general rank invariant-zigzag} 
    \item (Multiparameter persistence) Consider the $2$-dimensional grid poset $\Z^2$ and a functor $F:\Z^2\rightarrow \vect$. For any $(a,b),(c,d)\in \Z^2$ with $(a,b)\leq (c,d)$, the rank of $\varphi_F((a,b),(c,d))$ coincides with the rank of the restricted diagram $F|_{[(a,b),(c,d)]}$, in the sense of Definition \ref{def:fully general rank}. However, the standard rank invariant \cite{carlsson2009theory} does not record the rank of $F|_{I}$ when $I$ is not in the form of $[(a,b),(c,d)]$. See Figure \ref{fig:refinement} for an illustrative example. This implies that the generalized rank invariant is a refinement of the standard rank invariant for 2-parameter persistence modules $\Z^2\rightarrow \vect$. Similar argument applies to diagrams $\Z^d\rightarrow \vect$ for $d> 2$.\label{item:fully general rank invariant-multiparameter} 
\end{enumerate}
\end{example}

\begin{figure}
    \centering
    \includegraphics[width=0.35\textwidth]{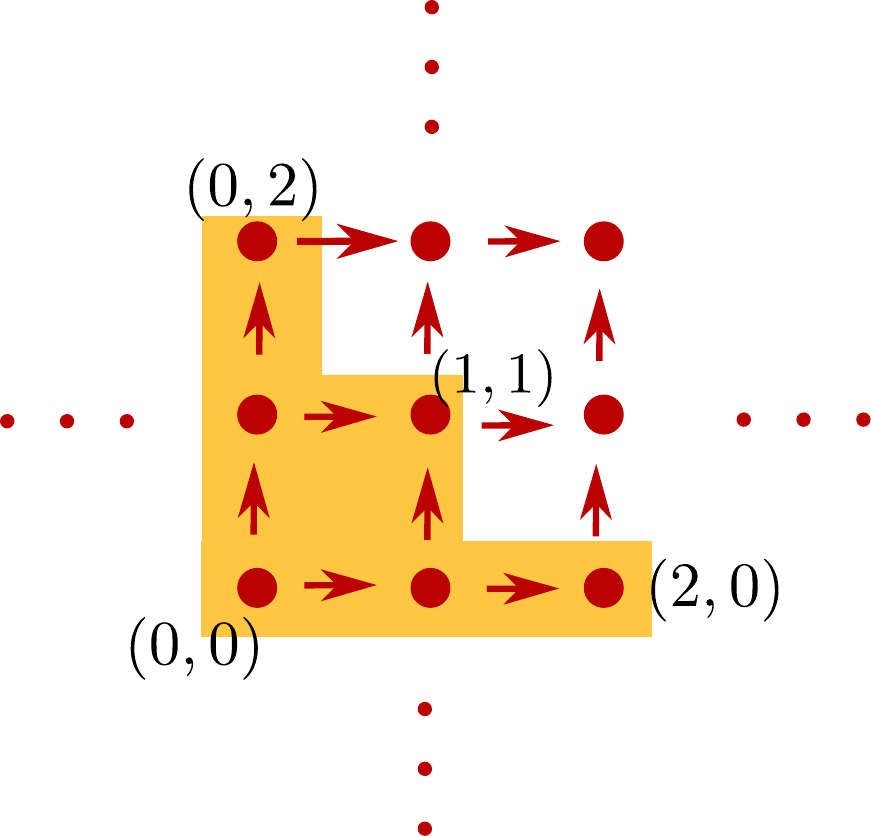}
    \caption{Consider $I\in\con(\Z^2)$ consisting of 6 points in $\Z^2$, depicted as above. The standard rank invariant of $F:\Z^2\rightarrow \vect$ does not record the rank of $F|_I$, whereas the generalized rank invariant of $F$ does.}
    \label{fig:refinement}
\end{figure}

\paragraph{Order-reversing property of the rank invariant.} It is well-known that there exists a translation-invariant order $\leq $ on the Grothendieck group $\A(\C)$ of $\C$ \cite[Section 6.1]{patel2018generalized}, \cite{weibel2013k}. For example, both $(\A(\vect),\leq)$ and $(\A(\sets),\leq)$ are isomorphic to $(\Z,\leq)$.  In analogy with the case of standard persistent modules\footnote{Given a diagram $F:\Z\rightarrow \vect$, it holds that
$\rank \ \varphi_F(i',j')\leq \rank \ \varphi_F(i,j)$  for $i'\leq i \leq j \leq j'$ in $\Z$.
}, if we assume that the target $\C$ satisfies certain properties, then the generalized rank invariant is also order-reversing. To establish this result, we first introduce a proposition which is a slight extension of \cite[Lemma 3.4]{puuska2017erosion}:

\begin{proposition}\label{lem:order reversing property}Let $\C$ be an essentially small, symmetric monoidal category with \emph{epimorphic} images (Definition \ref{def:images}). Also, assume that, for all $a,b\in \ob(\C)$, 
\begin{enumerate}[label=(\roman*)]
    \item $a\hookrightarrow b\Rightarrow [a]\leq [b]$ in $\A(\C)$,\label{item:injection}
    \item $a\twoheadrightarrow b \Rightarrow [b]\leq [a]$ in $\A(\C)$.\label{item:surjection}
\end{enumerate}
Let $f: a \rightarrow b$, $g: a' \rightarrow a$, $h: b \rightarrow b'$ be morphisms in $\C$ and denote $f' = h\circ f\circ g.$ Then, $\left[\im(f')\right] \leq \left[\im (f)\right]$ in $\A(\C)$.
\end{proposition}

\begin{proof}By the universal properties of images (which are epimorphic), we have the following commutative diagram:
\begin{equation}\label{eq:big rank diagram}
\begin{tikzcd}
  a' \arrow[rrrrr, "f'"]\arrow[dr, "g"]\arrow[dddrrr, two heads, to path=|- (\tikztotarget)]&&&&& b'\\
  & a \arrow[rr, "f"]\arrow[d, two heads]&& b\arrow[urr, "h"]\\
  & \im(f) \arrow[rr, two heads, "e"']\arrow[urr, hook, swap, "m_1"] &&\im\left(h\circ m_1\right) \arrow[uurr, hook]\\
  &&& \im(f')\arrow[u, hook, "m_2"]\arrow[uuurr, hook, to path=-| (\tikztotarget)]
\end{tikzcd}
\end{equation}
In particular, by the epimorphism $e$ and the monomorphism $m_2$ (Remark \ref{rem:monic}) in the diagram, we have $\left[\im(f')\right]\leq \left[\im\left(h\circ m_1\right) \right]\leq \left[\im (f)\right]$, as desired.
\end{proof}

\begin{proposition}[Order-reversing property of the generalized rank invariant]\label{prop:order reversing property}
Let $\C$ be an essentially small, symmetric monoidal category satisfying the conditions in Convention \ref{convention} \ref{item:a full subcategory of a bicomplete category}. Let $\Pb$ be any locally finite, connected poset and let $F:\Pb\rightarrow \C$ be a functor. Then, whenever $I$ and $I'$ in $\con(\Pb)$ are such that $I\subset I'$, it holds that $\rk(F)(I')\leq \rk(F)(I)$ in $\A(\C)$.  
\end{proposition}

The order-reversing property of the rank invariant can sometimes be useful for the efficient computation of persistence diagrams we will define (Definition \ref{def:generalized PD2} and Example \ref{ex:Amit}). 

\begin{proof}[Proof of Proposition \ref{prop:order reversing property}]Fix $I,I'\in \con(\Pb)$ such that $I\subset I'$. Notice that the limit $\left(\varprojlim F|_{I'},(\pi'_a)_{a\in I'}\right)$ and colimit $\left(\varinjlim F|_{I'},(i'_a)_{a\in I'}\right)$ can be restricted to a cone and a cocone over $F|_{I}$, respectively (Remark \ref{rem:restriction}). Then by the terminal property of $\left(\varprojlim F|_{I},(\pi_a)_{a\in I}\right)$ and the initial property of $\left(\varinjlim F|_{I},(i_a)_{a\in I}\right)$, there exist unique morphisms $p:\varprojlim  F|_{I'}\rightarrow \varprojlim F|_{I}$ and $\iota:\varinjlim F_{I}\rightarrow \varinjlim F|_{I'}$ such that for every $a$ in $I$, $\pi_a'=\pi_a\circ p$ and $i_a'=\iota\circ i_a$.
Let $\psi_F(I):\varprojlim F|_{I}\rightarrow \varinjlim F|_{I}$ and  $\psi_M(I'):\varprojlim F|_{I'}\rightarrow \varinjlim F|_{I'}$ be the canonical LC maps. Fix $a\in I$. Then \[\psi_M(I')=i'_a\circ \pi'_a = (\iota\circ i_a)\circ(\pi_a\circ p)=\iota\circ (i_a\circ \pi_a) \circ p=\iota\circ \psi_M(I)\circ p.\]

We first show that $\im(\psi_F(I))$ belongs to $\ob(\C)$, not $\ob(\D)\setminus \ob(\C)$. 
By the universal property of $\im(i_a\circ \pi_a)(=\im(\psi_F(I)))$ and $\im(i_a)$, there exist monomorphisms $\im(i_a\circ \pi_a)\hookrightarrow \im(i_a)\hookrightarrow \varinjlim F|_I$ (Remark \ref{rem:monic}). Also, since $\C$ has epimorphic images, we have an epimorphism from $F_a\in \ob(\C)$ to $\im(i_a)$. Then, by assumption \ref{item:epi}, $\im(i_a)\in \ob(\C)$. Now, by assumption \ref{item:mono}, $\im(i_a\circ p_a)\in \ob(\C)$. By the same argument, one can check that $\im(\psi_F(I'))\in \ob(\C)$. 

Note that $\iota$ and $p$ are morphisms of $\D$, but not necessarily $\C$. However, in diagram (\ref{eq:big rank diagram}),  by replacing $f,f',g$ and $h$ by $\psi_M(I),\psi_M(I'),p$ and $\iota$ respectively, and invoking assumptions \ref{item:mono} and \ref{item:epi} in Proposition \ref{prop:order reversing property}, we have $\rk(F)(I')\leq \rk(F)(I)$ in $\A(\C)$ as desired. 
\end{proof}

\subsection{Persistence diagrams for persistence modules over posets}\label{subsec:rank invariant in general}

 In this section we define the generalized persistence diagram of a functor $\Pb\rightarrow \C$ where $\Pb$ is a connected and \emph{essentially finite} poset (defined below), and $\C$ satisfies the conditions in Convention \ref{convention}.
 
\paragraph{Essentially finite posets.} 
A locally finite poset $\Pb$ is said to be \emph{essentially finite} if every  $I\in \con(\Pb)$ has a finite \emph{perimeter}:   

\begin{definition}[Neighborhood and perimeter]\label{def:perimeter} Let $\Pb$ be a locally finite poset. For $I\in \con(\Pb)$, we define the \emph{neighborhood} of $I$ as
\[\nbd(I):=\{p\in \Pb\setminus I: \mbox{there exists}\  q\in I \,\,\mbox{such that}\,\,q\asymp p\}.
\]
The \emph{perimeter} $o_{I}$ of $I$ is defined as the cardinality of $\nbd(I)$ (note that if $I$ is a singleton $\{p\}$, then $\nbd(I)$ is the neighborhood of the vertex $p$ in the graph $\Hasse(\Pb)$). 
\end{definition}

We remark that, for each interval $I$ of $\Z$, the perimeter of $I$ is either $0,1$ or $2$: If $I=\Z$, then the perimeter is $0$. For intervals of the form $I=[b,\infty)$, $b\in\Z$ or $I=(-\infty,d]$, $d\in \Z$, the perimeter of $I$ is $1$.  If $I=[b,d]$ for some $b,d\in \Z$, then perimeter of $I$ is $2$. Similarly, for the poset $\ZZ$, the perimeter of any interval of $\ZZ$ is either $0,1$ or $2$.

\begin{definition}[Essentially finite posets]\label{def:essentially} A poset $\Pb:=(\Pb,\leq)$ is said to be \emph{essentially finite} if 
\begin{enumerate}[label=(\roman*)]
    \item $(\Pb,\leq)$ is locally finite, and
    \item For each $I\in \con(\Pb)$, the perimeter $o_I$ of $I$ is finite.
\end{enumerate}
\end{definition}

Examples of essentially finite posets include all finite posets, the integers $\Z$, and $\ZZ$. For $d>1$, the infinite $d$-dimensional integer grid $\Z^d$  is locally finite, but \emph{not} essentially finite. For example, the interval $\Z\times\{0\}$ of $\Z^2$ has the infinite neighborhood $\Z\times \{1,-1\}$.

\paragraph{Persistence diagrams for generalized persistence modules.} The following notation is useful for defining generalized persistence diagrams.

\begin{notation}[$n$-th entourage]\label{nota:I^n} Let $\Pb$ be a connected, essentially finite poset and let $I\in \con(\Pb)$. Fix $n\in \N$. By $I^n$, we denote the set of all $J\subset \Pb$ such that $I\subset J\subset I\cup\nbd(I)$ and $\abs{J\cap \nbd(I)}=n$. In other words, each $J\in I^n$ is obtained by adding $n$ points of $\nbd(I)$ to $I$. We refer to $I^n$ as the \emph{$n$-th entourage} of $I$.
\end{notation}

\begin{Remark}  Let $o_I$ be the perimeter of $I$ (Definition \ref{def:perimeter}).
\begin{enumerate}[label=(\roman*)]
\item For example, $I^1 = \{I\cup\{p\}:p \in \nbd(I) \},$ and $\abs{I^1} = o_I$. In general, for $n\leq o_I$, one has $\abs{I^n} = \binom{o_I}{n}$, whereas for $n>o_I$, $I^n=\emptyset$.
  
\item  Any  $J\in I^n$ belongs to $\con(\Pb)$.
\end{enumerate}
\end{Remark}

By Remark \ref{rem:counterpart}, we can identify the Grothendieck group $\A(\vect)$ with the integer group $(\Z,+)$. Given an $F:\Z\rightarrow \vect$ and $[a,b]\in \inter(\Z)=\con(\Z)$ with $a<b$, let us recall that in the standard persistence diagram of $F$ the multiplicity of $(a,b)\in \Z^2$  is defined as \cite{cohen2007stability}:
\[\rk(F)([a,b])-\rk(F)([a-1,b])-\rk(F)([a,b+1])+\rk(F)([a-1,b+1]).
\]

This formula is a motivation for:

\begin{definition}[Generalized persistence diagram of  $\Pb\rightarrow \C$]\label{def:generalized PD2} Let $\Pb$ be a connected, essentially finite poset. Given an $F:\Pb\rightarrow \C$, let us define the persistence diagram $\dgm^{\Pb}(F):\con(\Pb)\rightarrow \A(\C)$ of $F$ by sending each $I\in \con(\Pb)$ to
\begin{equation}\label{eq:mobius formula}
    \dgm^{\Pb}(F)(I):=\rk(F)(I)-\sum_{J\in I^1}\rk(F)(J)+\sum_{K\in I^2}\rk(F)(K)-\ldots+(-1)^{o_I}\sum_{L\in I^{o_I}}\rk(F)(L).
\end{equation}
\end{definition}

At the end of this section we will show that $\dgm^{\Pb}(F)$ is the M\"obius inversion of $\rk(F)$ over the poset $\con^{\mathrm{op}}(\Pb)$, generalizing the framework of \cite{patel2018generalized}. Now, by invoking the principle of inclusion-exclusion, we prove that the barcode of interval decomposable persistence modules (Definition \ref{def:barcode}) can be obtained via formula (\ref{eq:mobius formula}): 

\begin{theorem}[Generalized persistence diagram recovers the barcode]\label{prop:completeness} Let $\Pb$ be any connected, essentially finite poset and let  $F:\Pb\rightarrow \vect$ be interval decomposable.  For  $I\in \con(\Pb)$, the value given by equation (\ref{eq:mobius formula}) is equal to the multiplicity of $I$ in $\barcp(F)$ (this implies that if $I\in \con(\Pb)\setminus \inter(\Pb)$, then the value is zero).
\end{theorem}

See Figure \ref{fig:inclusion-exclusion} for illustrative examples of applications of Theorem \ref{prop:completeness}. We remark that formula (\ref{eq:mobius formula}) can  be simplified in special cases, e.g. see \cite{botnan2020rectangle} for the case of rectangle decomposable $\Z^2$-indexed persistence modules. We will obtain a proof of Theorem  \ref{prop:completeness} via an elementary argument resting upon the inclusion-exclusion principle. This proposition can also be obtained as a corollary to Propositions \ref{prop:original mobius} and \ref{prop:coincide}, based on standard results from incidence algebra.    

\begin{figure}
    \centering
    \includegraphics[width=0.9\textwidth]{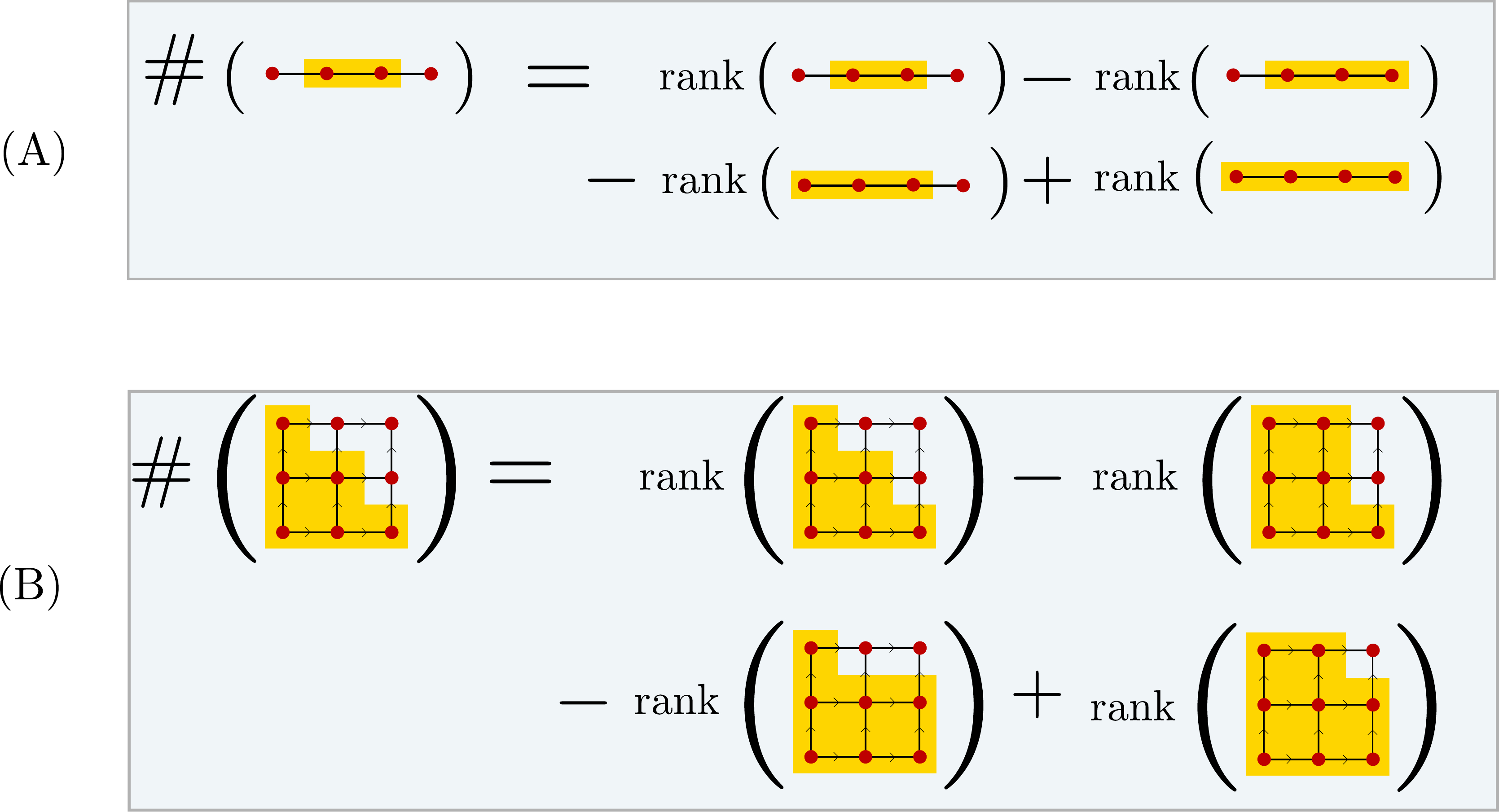}
    \caption{(A) Let $\leftrightarrow$ denote $\leq$ or $\geq$. Given any $\vect$-diagram $M$ indexed by $\{1\leftrightarrow 2\leftrightarrow 3\leftrightarrow 4\}$, $M$ is interval decomposable. The multiplicity of the interval $[2,3]$  in the barcode of $M$ is equal to $\mathrm{rank}(M|_{[2,3]})-\mathrm{rank}(M|_{[2,4]})-\mathrm{rank}(M|_{[1,3]})+\mathrm{rank}(M|_{[1,4]})$. (B) Given any interval decomposable persistence module $N$ over the grid $\{1,2,3\}\times\{1,2,3\}$, the multiplicity of the interval  indicated by the LHS in the barcode of $N$ can be computed in a similar way.}
    \label{fig:inclusion-exclusion}
\end{figure}

\begin{Remark}\label{rem:rk is complete}
Since the barcode of an interval decomposable $F:\Pb\rightarrow \vect$ is a complete invariant of $F$, Theorem \ref{prop:completeness} implies that the generalized rank invariant of $F$ (Definition \ref{def:fully generalized rank invariant}) is also a complete invariant for $F$.
\end{Remark}

\begin{Remark}\label{rem:implication2} Theorem \ref{prop:completeness} also implies the following:
Given any $G:\Pb\rightarrow \vect$, (1) if there exists $I\in \con(\Pb)\setminus \inter(\Pb)$ such that $\dgm^{\Pb}(G)(I)\neq 0$, then $G$ is \emph{not} interval decomposable. (2) If there exists $I\in \inter(\Pb)$ such that $\dgm^{\Pb}(G)(I)<0$, then $G$ is \emph{not} interval decomposable: See Example \ref{ex:Amit}. (3) The converse of statement in item (2) is not true: See Example \ref{ex:Alex}.
\end{Remark}

Let us recall the following folklore fact for $\Z$-indexed modules: Let $F:\Z\rightarrow \vect$ and let $a,b\in \Z$ with $a\leq b$. Then, the rank of $\varphi_F(a, b)$ is equal to the total number of intervals $J\in \barc^{\Z}(F)$ such that $[a,b]\subset J$. In our setting, this fact generalizes to:

\begin{proposition}[Barcode recovers the generalized rank invariant]\label{rem:implication} Let $\Pb$ be a connected, locally finite poset. For any interval decomposable $F:\Pb\rightarrow \vect$ and $J\in \con(\Pb)$, $\rk(F)(J)$ is equal to the total multiplicity of intervals $K\in\barcp(F)$ such that $K\supset J$. 
\end{proposition}
Here we provide a succinct proof of Proposition \ref{rem:implication} \emph{under the extra assumption} that $\varprojlim F|_J$ is of finite dimension. A general proof is deferred to the appendix due to its length (Section \ref{sec:proof}).

\begin{proof}[Proof in the case where $\varprojlim F|_J$ is finite dimensional] 
Since $F$ is interval decomposable, there exists an indexing set $C$ such that $F\cong \bigoplus_{c\in C} I^{K_c}$, where each $K_c$ is an interval of $\Pb$. Let us recall that (1) colimits preserve direct sums, and (2) limits preserve direct products \cite[Theorem V.5.1]{mac2013categories}. By virtue of the extra assumption that $\varprojlim F|_J$ is finite dimensional (as well as for each $p\in \Pb$, $F_{p}$ is finite dimensional), the notions of direct product and direct sum coincide in the category of cones over $F$. Therefore, 
\[\rk(F)(J)=\rank(F|_J)=\rank\left(\bigoplus_{c\in C} I^{K_c}|_J\right)=\sum_{c\in C}\rank\left(I^{K_c}|_{J}\right).
\]
Note that, for every $c\in C$, $
\rank\left(I^{K_c}|_{J}\right)=\begin{cases}1,&  K_c \subset J\\ 0,&\mbox{otherwise.} \end{cases}$ (cf. Remark \ref{rem:rank is finite}). Therefore, the right-most summation is the total multiplicity of intervals $K_c\in\barcp(F)$ such that $K_c\supset J$.
\end{proof}

\begin{proof}[Proof of Theorem \ref{prop:completeness}] Let $\barcp(F)=\lmulti J_c: F\cong \bigoplus_{c\in C} I^{J_c}\rmulti$ for some indexing set $C$. Let $a_0(I)$ be the cardinality of the set $\{c\in C: I\subset J_c\}$. Also, let $a_1(I)$ be the cardinality of the set $\{c\in C: I\subsetneq J_c\}.$ Since the difference $a_0(I)-a_1(I)$ is the multiplicity of $I$ in $\barcp(F)$, it suffices to show that $a_0(I)-a_1(I)$ is identical to the value given by formula (\ref{eq:mobius formula}). First, by Proposition \ref{rem:implication},
$a_0(I)=\rk(F)(I)$.

Next we compute $a_1(I)$. For any $K\in \con(\Pb)$, let 
$B_F(K):=\left\{ c\in C: K\subset J_c \right\}.$

Let $K\in \con(\Pb)$ be such that $I\subsetneq K$. Let us show that $K$ must include at least one element of $\nbd(I)$: Pick any $k\in K\setminus I$ and any $p\in I$. Since the induced subgraph of $\Hasse(\Pb)$ on $K$ is connected, there exist $p=q_0,q_1,\ldots,q_m=k$ in $K$ such that $q_i$ and $q_{i+1}$ are adjacent (i.e. $q_i\asymp q_{i+1}$) in the subgraph for all $i$. Since $p=q_0\in I$ and $q_m=k\in K\setminus I$, there must be $i\in\{1,\ldots,m\}$ such that $q_i$ belongs to $\nbd(I)$. 

Now, we have the equality \[\{c\in C: I\subsetneq J_c\}=\bigcup_{K\in I^1} B_F(K)\ \mbox{(see Notation \ref{nota:I^n})},\] and hence $a_1(I)=\abs{\bigcup_{K\in I^1} B_F(K)}$. In particular, since $F$ is a diagram of finite dimensional vector spaces,  the set $\{c\in C: I\subsetneq J_c\}$ and all the sets $B_F(K)$ must be finite sets. Therefore, by  the principle of inclusion-exclusion and Proposition \ref{rem:implication}, we have: %
\begin{align*}
    a_1(I)&=\sum_{J\in I^1}\abs{B_F(J)}-\sum_{K\in I^2}\abs{B_F(K)}+\ldots+(-1)^{o_I-1}\sum_{L\in I^{o_I}}\abs{B_F(L)}
    \\&=\sum_{J\in I^1}\rk(F)(J)-\sum_{K\in I^2}\rk(F)(K)+\ldots+(-1)^{o_I-1}\sum_{L\in I^{o_I}}\rk(F)(L),
\end{align*}
completing the proof.
\end{proof}

\paragraph{M\"obius function and another interpretation of Definition \ref{def:generalized PD2}.} Let $\Q$ be a locally finite poset. The M\"obius function $\mu_{\Q}:\Q\times\Q\rightarrow \Z$ of $\Q$ is defined\footnote{More precisely, the codomain of $\mu_{\Q}$ is the multiple of $1$ in a specified base ring rather than $\Z$.} recursively \cite{rota1964foundations} as

\begin{equation}\label{eq:induction}
    \mu_{\Q}(p,q)=\begin{cases}1,& \mbox{$p=q$,}\\ -\sum_{p\leq r< q}\mu_{\Q}(p,r),& \mbox{$p<q$,} \\ 0,& \mbox{otherwise.} \end{cases}
\end{equation}

\begin{proposition}[M\"obius inversion formula {\cite[Proposition 2 (p.344)]{rota1964foundations}}]\label{prop:original mobius} Let $\Q$ be a locally finite poset and let $k$ be a field. Suppose that an element $0\in \Q$ exists with the property that $0\leq q$ for all $q\in \Q$. Consider a pair of functions $f,g:\Q\rightarrow k$ with the property that
\[g(q)=\sum_{r\leq q} f(r).\]
Then, $\displaystyle f(q)=\sum_{r\leq q} g(r)\cdot \mu_{\Q}(r,q)$ for $q\in \Q$.
\end{proposition}

Under the assumption that $\con(\Pb)$ is locally finite, formula (\ref{eq:mobius formula}) is in fact the \emph{M\"obius inversion} of the rank invariant $\rk(F)$ over the poset $\con^{\mathrm{op}}(\Pb)$: 

\begin{proposition}[Extension of Patel's generalized persistence diagrams]\label{prop:coincide} Let $\Pb$ be a connected, essentially finite poset such that $(\con(\Pb),\subset)$ is locally finite. Let $F:\Pb\rightarrow \C$. Let $\mu:\con^{\mathrm{op}}(\Pb)\times \con^{\mathrm{op}}(\Pb)\rightarrow \Z$ be the M\"obius function of the poset $\con(\Pb)^{\mathrm{op}}$. 

Then, for $I\in \con^{\mathrm{op}}(\Pb)$
\[\dgm^{\Pb}(F)(I)= \sum_{J\supset I}\rk(F)(J)\cdot\mu(J,I).\]
\end{proposition}

We prove Proposition \ref{prop:coincide} at the end of this section.

\begin{example}\label{ex:Amit} Let $\Pb:=\{a,b,c,d\}$ equipped with the partial order $\leq:=\{(a,b),(c,b),(d,b)\}\subset \Pb\times \Pb$. Then, $\Hasse(\Pb)$ is shown as in (A) below. Consider $F:\Pb\rightarrow \vect$ given as in (B) below,
\begin{center}
\begin{tikzcd}
  & a \arrow[d, dash] \\
    & b  \\
  c \arrow[ru, dash] && d \arrow[lu, dash]\\
  & \mbox{(A)}
 \end{tikzcd}
 \hspace{15mm}
 \begin{tikzcd}
  & k \arrow[d, hook, "i_1"] \\
    & k^2  \\
  k \arrow[ru, hook, "i_2"] && k \arrow[lu, hook, "i_1+i_2"]\\
  & \mbox{(B)}
 \end{tikzcd}
\end{center}
where $i_1,i_2:k\rightarrow k^2$ are the canonical inclusions into the first factor and the second factor of $k^2$, respectively. We show that $F$ is \emph{not} interval decomposable.

By Remark \ref{rem:implication2}, it suffices to show that, for $I:=\{b\}$, one has $\dgm^\Pb(F)(I)<0$. Observe that $I^1=\{\{a,b\},\{b,c\},\{b,d\}\}$, $I^2=\{\{a,b,c\},\{a,b,d\},\{b,c,d\}\}$, and $I^3=\{\{a,b,c,d\}\}=\{\Pb\}$, $I^n=\{\emptyset\}$ for $n>3$. Note that:

\begin{itemize}
    \item $\rk(F)(I)=2$, which is the dimension of $F(b)=k^2$ (cf. Example \ref{ex:trivial cases} \ref{item:trivial cases1}).
    \item $\rk(F)(J)=1$ for every $J\in I^1$ (cf. Example \ref{ex:trivial cases} \ref{item:trivial cases2}).
    \item $\rk(F)(J)=0$ for every $J\in I^2$: the sub-diagram $F|_J$ amounts to a zigzag module of length 3 \cite{carlsson2009zigzag}. It is not difficult to check that the barcode of $F|_J$ does not include the full interval $J$. This implies, by Proposition \ref{rem:implication}, $\rk(F)(J)=\rk(F|_J)(J)=0$. 
    \item $\rk(F)(\Pb)=0$: Fix $J\in I^2$. By Proposition \ref{prop:order reversing property}, $0\leq \rk(F)(\Pb)\leq \rk(F)(J)=0$.
\end{itemize}
Therefore, formula (\ref{eq:mobius formula}) gives that $\dgm^{\Pb}(F)(I)=2-(1+1+1)+(0+0+0)-0=-1$, completing the proof.
\end{example}

The non-negativity of the persistence diagram of a $G:\Pb\rightarrow \vect$ does not imply the interval-decomposability of $G$:

\begin{example}\label{ex:Alex} Let $\Pb$ and $F:\Pb\rightarrow \vect$ respectively be the same as Example \ref{ex:Amit}. Let $G$ be the direct sum of $F$ and the interval module $\mathcal{J}^{\{b\}}:\Pb\rightarrow \vect$ supported by $\{b\}$. Note that $\rk(G)=\rk(F)+\rk(\mathcal{J}^{\{b\}})$ and
\[\dgm^{\Pb}(G)(I)=\begin{cases}0,& I=\{b\} \\ \dgm^{\Pb}(F)(I),&\mbox{otherwise.}
\end{cases}
\]
In particular, observe that $\dgm^{\Pb}(G)$ is non-negative even though $G$ is not interval decomposable. Also, note that, for $H=I^{\{a,b\}}\bigoplus I^{\{b,c\}}\bigoplus I^{\{b,d\}}$, we have $\rk(G)=\rk(H)$ and in turn $\dgm^{\Pb}(G)=\dgm^{\Pb}(H)$. This shows that the generalized rank invariant and persistence diagram are \emph{not} complete invariants beyond the class of interval decomposable persistence modules. The following is also noteworthy: for the intervals $I=\{a,b\},\ \{b,c\},\ \{b,d\}$,   $G$ does not admit $\mathcal{J}^I$ as a summand even though $\dgm^{\Pb}(G)(I)=1$.
\end{example}

\begin{Remark} Given $F:\Pb\rightarrow \vect$, for each $K\in \inter(\Pb)$, let $n_K:=\dgm^{\Pb}(F)(K)$. When $\dgm^{\Pb}(F)$ is non-negative on $\inter(\Pb)$, the direct sum $G=\bigoplus_{K\in \inter(\Pb)}n_K\cdot I^K$ of interval modules could be seen as an ``approximation" of $F$. In particular, by Theorem \ref{prop:completeness} we know that $F\cong G$ whenever $F$ is interval decomposable. This observation is connected to \cite{asashiba2019approximation}, where the authors propose a method to approximate a diagram $F:\R^2\rightarrow \vect$ by an interval decomposable  $\delta^\ast(F):\R^2\rightarrow \vect$ whose (standard) rank invariant is the same as $F$.  In particular, $F=\delta^*(F)$ whenever $F$ is interval decomposable.
\end{Remark}

\paragraph{Towards a proof of Proposition \ref{prop:coincide}.} In order to prove Proposition \ref{prop:coincide}, we will exploit the poset structure of $\con^{\mathrm{op}}(\Pb)$. To this end, we briefly review the notion of \emph{lattice}. See  \cite{birkhoff1940lattice} for a comprehensive treatment on this subject. 

Let $\Lb$ be a poset, and let $S \subset \Lb$. An element $u \in \Lb$ is said to be an upper bound of $S$ if $s \leq u$ for all $s \in S.$ An upper bound $u$ of $S$ is said to be its least upper bound, or \emph{join} if $u \leq x$ for each upper bound $x$ of $S$. If a join of $S$ exists, then it is unique. The concepts of lower bound and greatest lower bound are defined in a dual way. In particular, a greatest lower bound is said to be a \emph{meet}.

\begin{definition}[Lattices]\label{def:lattice}Let $\Lb$ be a poset. $\Lb$ is said to be a \emph{join-semi lattice} if for every pair $a,b\in \Lb$, the set $\{a,b\}$ has a join in $\Lb$. Dually, $\Lb$ is said to be a \emph{meet-semi lattice} if for every pair $a,b\in \Lb$, the set $\{a,b\}$ has a meet in $\Lb$. If $\Lb$ is both join-semi lattice and meet-semi lattice, then $\Lb$ is said to be a \emph{lattice}.
\end{definition}

\begin{lemma}[$\con^{\mathrm{op}}(\Pb)$ is locally a lattice]\label{lem:lattice}
Let $\Pb$ be a connected, essentially finite poset such that $\con(\Pb)$ is locally finite. Let $I,J\in \con^{\mathrm{op}}({\Pb})$ with $J\supset I$. Then $[J,I]:=\{K\in \con^{\mathrm{op}}(\Pb): J\supset K \supset I\}$ is a finite lattice.
\end{lemma}

\begin{proof}
By assumption, $\con(\Pb)$ is locally finite and thus $[J,I]$ is finite. Fix any $K_1,K_2 \in [J,I]$. Observe that $K_1\cup K_2$ is the greatest lower bound for $\{K_1,K_2\}$ in the subposet $[J,I]$ of $\con^{\mathrm{op}}(\Pb)$. Also, observe that the least upper bound for $\{K_1,K_2\}$ is the union of all subposets $L$ in the subcollection
$\{L\in \con^{\mathrm{op}}(\Pb):  K_1\cap K_2 \supset L \supset I \}$ of $[J,I]$ (note that this subcollection contains $I$ and hence not empty).
This shows that $[J,I]$ is a lattice.
\end{proof}

An \emph{atom} in a poset is an element that covers a minimal element. A \emph{dual atom} is an element that is covered by a maximal element \cite{rota1964foundations}. 

\begin{proposition}[{\cite[Corollary (P. Hall), p.349]{rota1964foundations}}]\label{prop:Hall}Let $\Lb$ be a finite lattice with $0,1\in \Lb$ such that $0\leq l \leq 1$ for all $l\in \Lb$. If $0$ is not the meet of dual atoms of $\Lb$, or if $1$ is not the join of atoms, then $\mu_{\Lb}(0,1)=0$.
\end{proposition}

Let $\Pb$ be a connected, essentially finite poset such that $\con(\Pb)$ is locally finite. Let $I,J\in \con(\Pb)$  with $J\supset I$. Observe that, in the finite lattice $\Lb=[J,I]$ which is a subposet of $\con^{\mathrm{op}}(\Pb)$, all dual atoms belong to the first entourage $I^1$ of $I$ (Notation \ref{nota:I^n}).

\begin{proof}[Proof of Proposition \ref{prop:coincide}]
Invoking formula (\ref{eq:induction}), by elementary induction, it follows that for any $n\in\{1,\ldots,o_I\}$ and  for any $J\in \con^{\mathrm{op}}(\Pb)$ with $J\in I^n$, $\mu(J,I)=(-1)^n$. 

Now pick any $J\in \con^{\mathrm{op}}(\Pb)$ such that $J\supset I$ and $J\notin \bigcup_{n=1}^{o_I}I^n$. This implies that $J$ is not the union of elements in $I^1$ and in turn implies that $J$ is not the meet of dual atoms of $[J,I].$ By \cite[Proposition 4  (p.345)]{rota1964foundations}, the restriction of the M\"obius function of $\con^{\mathrm{op}}(\Pb)$ to $[J,I]\times [J,I]$ is equal to the M\"obius function of the subposet $[J,I]$. Lemma \ref{lem:lattice} and Proposition \ref{prop:Hall} now directly imply that $\mu(J,I)=0$. 
\end{proof}

\section{Computing the 0-th level set barcode of a Reeb graph within $\sets$-category}\label{sec:reeb graph}

In this section we propose a novel method to compute the $0$-th level set barcode of a Reeb graph within the category of sets (Section \ref{sec:reeb graph barcode}). This will be possible via a geometric interpretation of $\ZZ$-indexed $\sets$-diagrams (Section \ref{sec:local analysis of reeb graphs}). An improvement of a semicontinuity theorem by Patel will also follow thereby (Section \ref{sec:set-stability}).  %

\subsection{Local analysis of Reeb graphs}\label{sec:local analysis of reeb graphs}

\paragraph{Reeb graphs.} Let $f:M\rightarrow \R$ be a Morse(-type) function (Definition \ref{def:morse type}) on a manifold. The Reeb graph of $f$ is a descriptor for the evolution of connected components of the level sets $f^{-1}(r)$ as $r\in \R$ varies \cite{reeb1946points}. There have been not only numerous applications of Reeb graphs in shape analysis and visualization \cite{BGSF08,shinagawa1991surface} or in dynamic/high-dimensional data analysis \cite{buchin2013trajectory,ge2011data,kim2017stable,kim2018CCCG}, but also a vast body of theoretical study on Reeb graphs \cite{de2016categorified,stefanou2019reeb}, approximation/computation of Reeb graphs \cite{dey2013reeb,harvey2010randomized,Salman12}, metrics on Reeb graphs \cite{bauer2014measuring,bauer2015strong,bauer2018reeb,carriere2017local}, and generalizations \cite{dey2016multiscale,mapper}. 

By a Reeb graph, we will refer to  a \emph{constructible} topological graph $X$ equipped with a notion of ``height" represented by a continuous function $f:X\rightarrow \R$. While we defer the rigorous definition of Reeb graphs to Appendix (Definitions \ref{def:morse type} and \ref{def:Reeb graph}), it is well-known that any diagram $M:\ZZ\rightarrow \sets$ amounts to a Reeb graph and vice versa as we will see below in more detail \cite{curry2016classification,de2016categorified}.

\begin{definition}[Reeb-graph-realization of a zigzag diagram in $\sets$]\label{def:realization} Let $M:\ZZ\rightarrow \sets$. We define the \emph{Reeb graph corresponding to $M$} as the pair $(\reeb(M),\pi)$ consisting of a topological graph $\reeb(M)$ and a map $\pi:\reeb(M)\rightarrow \R$, described subsequently (referring to Example \ref{ex:Reeb}  may help the understanding of the description below).

\begin{enumerate}
    \item For   $(i,i)\in \ZZ$, let each element in $M_{(i,i)}$ become a vertex which lies over $i\in \Z (\subset \R)$. 
    \item For $(i,i-1) \in \ZZ$, let each element in $M_{(i,i-1)}$ become an edge which lies over the interval $[i-1,i]\subset \R$.
    \item For each comparable pair $(i,i-1)\leq (i,i)$ (resp.  $(i,i-1)\leq (i-1,i-1)$) in $\ZZ$, the attaching map between the vertex set and the edge set is specified by $\varphi_M((i,i-1), (i,i))$ (resp. $\varphi_M((i,i-1), (i-1,i-1))$).
 \end{enumerate}
 The space $\reeb(M)$ is the quotient of the disjoint union of the spaces $M_{(i,i)}\times \{i\}$  and $ M_{(i,i-1)}\times [i-1, i]$ for all $i\in \Z$ with respect to the identifications \[(\varphi_M((i,i-1),(i,i))(e),i)\sim (e,i) \ \mbox{ and } \  (\varphi_M((i,i-1),(i-1,i-1))(e),i-1)\sim (e,i-1).\] The map $\pi:\reeb(M)\rightarrow \R$ is defined as the projection onto the second factor.
 
For any interval $I$ of $\ZZ$, one can define the Reeb graph corresponding to a diagram $I\rightarrow \sets$ in the same way.
\end{definition}

\begin{example}\label{ex:Reeb}%
Consider $M:\ZZ\rightarrow \sets$ specified as follows:
\begin{equation*}
    M_{(1,1)}=\{v_1,v_2\},\hspace{3mm}M_{(2,1)}=\{e_1,e_2\},\hspace{3mm} M_{(2,2)}=\{v_3,v_4\},\hspace{3mm}M_{(3,2)}=\{e_3,e_4\},\hspace{3mm} M_{(3,3)}=\{v_5,v_6\},
\end{equation*}
and other $M_{(i,j)}$ are the empty set. The four maps \[\begin{tikzcd}M_{(1,1)}\arrow[<-]{rr}{\varphi_M ((2,1),(1,1))}&& M_{(2,1)}\arrow[->]{rr}{\varphi_M ((2,1),(2,2))}&&M_{(2,2)}\arrow[<-]{rr}{\varphi_M ((3,2),(2,2))}&& M_{(3,2)}\arrow[->]{rr}{\varphi_M ((3,2),(3,3))}&&M_{(3,3)}\end{tikzcd}\]
are defined as follows: 
\[\begin{tikzcd}v_1&&\arrow[mapsto]{ll}{} e_1\arrow[->,mapsto]{rr}{}&&v_3&&\arrow[mapsto]{ll}{} e_3\arrow[mapsto]{rr}{}&&v_5 \\
v_2&&\arrow[mapsto]{ll}{}e_2\arrow[->,mapsto]{rr}{}&&v_4&&\arrow[mapsto]{llu}{} e_4\arrow[mapsto]{rr}{}&&v_6.
\end{tikzcd}\]
The Reeb graph corresponding to $M$ is depicted in Figure \ref{fig:Reeb}.
\end{example}

\begin{figure}
    \centering
    \includegraphics[width=0.4\textwidth]{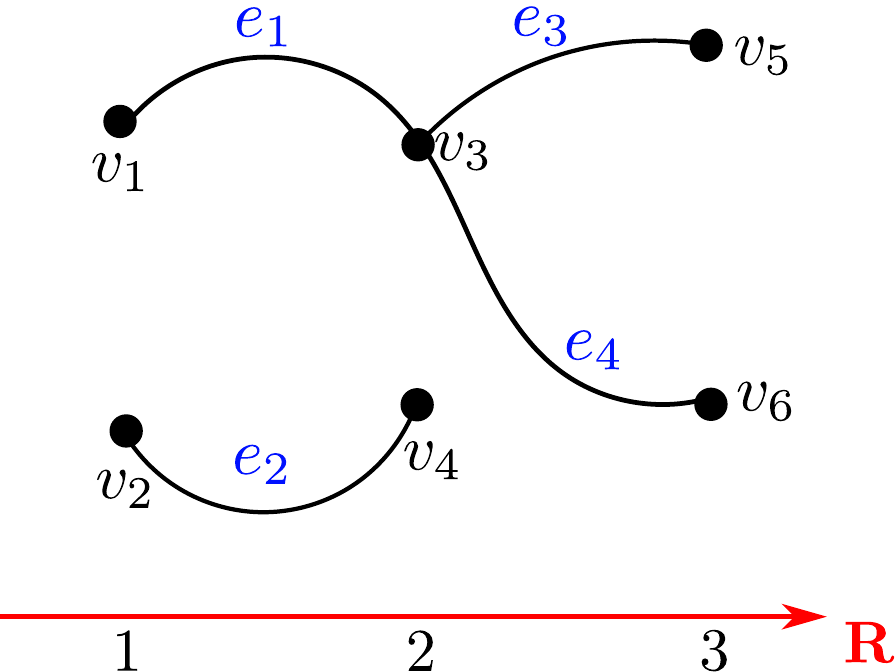}
    \caption{The Reeb graph $(\reeb(M),\pi)$ corresponding to $M$ in Example \ref{ex:Reeb}, where $\pi:\reeb(M)\rightarrow \R$ is the projection map to the horizontal real axis.}
    \label{fig:Reeb}
\end{figure}

Definition \ref{def:realization} describes how to turn a functor $M:\ZZ\rightarrow \sets$ into a Reeb graph. In this respect, we will refer to any $M:\ZZ\rightarrow \sets$ as a \emph{Reeb graph}.  On the other hand, any Reeb graph can be expressed as a $\sets$-valued  zigzag diagram over $\R$ \cite{curry2016classification,de2016categorified}. This zigzag diagram over $\R$ is in turn equivalent to a $\ZZ$-index $\sets$-diagram up to rescaling--- this idea  has already implicitly appeared in \cite{botnan2018algebraic,curry2013sheaves}. This means that the entire combinatorial information of a Reeb graph $X$ can be encoded into a certain $M:\ZZ\rightarrow \sets$. In particular, the $0$-th level set barcode of $X$ \cite{carlsson2009zigzag} can be extracted from $M$ up to rescaling of intervals.

\paragraph{Local analysis of Reeb graphs.} We illustrate how to extract topological/combinatorial information from limits and colimits of interval restrictions of a Reeb graph $M:\ZZ\rightarrow \sets$.  For $k,l\in \ZZ$, assume that $x\in M_{k}$ and $y\in M_{l}$. We write $x\sim y$ if $k$ and $l$ are comparable and one of $x$ and $y$ is mapped to the other via the internal map between $M_k$ and $M_l$.

 We can explicitly represent the limit and colimit of $M$ as follows. 
\begin{enumerate}[label=(\roman*)]
    \item The limit of $M$ is the pair $\left(L,(\pi_k)_{k\in\ZZ}\right)$, where
    \[L:=\left\{(x_k)_{k\in \ZZ}\in \prod_{k\in\ZZ} M_k: \ \mbox{for comparable $k,l\in\ZZ$,}\  x_{k}\sim x_{l} \right\},\]and each $\pi_k:L\rightarrow M_k$ is the canonical projection. \label{item:set-limit construction}
    \item The colimit of $M$ is the pair $\left(C,(i_k)_{k\in\ZZ}\right)$ described as follows:
    \begin{equation}\label{eq:set colimit construction}
        C:=\left(\coprod_{k\in\ZZ}M_k\right)\big/\approx,
    \end{equation}
    where $\approx$ is the equivalence relation generated by the relations $x_k\sim x_{l}$ for $x_k\in M_k$ and $x_{l}\in M_{l}$ with $k,l$ being comparable. For the quotient map $q:\coprod_{k\in\ZZ}M_k\rightarrow C$, each $i_k$ is the composition $M_k\hookrightarrow \coprod_{k\in\ZZ}M_k \stackrel{q}{\rightarrow} C$. \label{item:set-colimit construction}
\end{enumerate}

Let $I\in \inter(\ZZ)$. For any diagram $N:I\rightarrow \sets$, we can construct the limit and colimit of $N$ in the same way; namely, in items \ref{item:set-limit construction} and \ref{item:set-colimit construction} above, replace $M$ and $\ZZ$ by $N$ and $I$, respectively. \textbf{In what follows, we use those explicit constructions  whenever considering limits and colimits of (interval restrictions of) $\ZZ$-indexed $\sets$-diagrams.} 

\begin{definition}[Supports and full components]\label{def:full component} Let $I\in \inter(\ZZ)$ and let $N:I\rightarrow \sets$. Let $c\in \varinjlim N$. We define the \emph{support} of $c$ as
\[\supp(c):=\{k\in I:  \exists x_k\in N_k, \ i_k(x_k)=c\}.\]
In particular, if $\supp(c)=I$, we call $c$ a \emph{full component} of $N$. 
\end{definition}

Now, extending ideas in \cite{de2016categorified}, we geometrically interpret elements of limits, colimits and images of the canonical LC maps of (interval restrictions of) $\ZZ$-indexed $\sets$-diagrams. 

Recall that, for $b,d\in \Z\ \mbox{with $b< d$}$, $\langle b,d \rangle_{\ZZ}$ stands for an interval of $\ZZ$ (Notation \ref{nota:intervals of zz}). Similarly, for $b,d\in\R$ with $b< d$, $\langle b,d \rangle$ will stand for one of the real intervals $(b,d),[b,d],(b,d],[b,d)$. The terminology introduced in Definition \ref{def:full components} below will be illustrated in Example \ref{ex:Reeb2}.

\begin{definition}[Local analysis of a Reeb graph]\label{def:full components}
Consider any $M:\ZZ\rightarrow \sets$ and fix $I:=\langle b,d \rangle_{\ZZ} \in \inter(\ZZ)$. %
\begin{enumerate}[label=(\roman*)]
    \item Each element in  $\varinjlim M|_{I}$ is said to be a $\langle b,d \rangle$-component.
    \item Each element in  $\varprojlim M|_{I}$ is said to be a $\langle b,d \rangle$-section.\footnote{A right-inverse of a morphism is called a section \cite{mac2013categories}. See Example \ref{ex:Reeb2} \ref{item:Reeb2 2}.}\label{item:full components2}
    \item If $c\in \varinjlim M|_{I}$ and $\supp(c)=I$, then $c$ is said to be a $\langle b,d \rangle$-\emph{full}-component.
    \item Each element in the image of the canonical LC map $\psi_{M|_{I}}:\varprojlim M|_{I}\rightarrow \varinjlim M|_{I}$ is said to be a $\langle b,d \rangle$-full-component of $M$ \emph{with a section}.  \label{item:full components4} 
\end{enumerate}

\end{definition}

\begin{example}[Geometric interpretation of Definition \ref{def:full components}]\label{ex:Reeb2} 
For $M:\ZZ\rightarrow \sets$ in Example \ref{ex:Reeb} (Figure \ref{fig:Reeb}), note the following: 
\begin{enumerate}[label=(\roman*)]
    \item There are two $[1,3]$-components: The connected component containing the vertex $v_1$ and the one containing $v_2$. Also, there are two $[1,2]$-components: these components correspond to the restrictions of the previous two components to $\pi^{-1}([1,2])$. \label{item:I-component}
    \item The 5-tuples $(v_1,e_1,v_3,e_3,v_5)$ and $(v_1,e_1,v_3,e_4,v_6)$ are $[1,3]$-sections. These tuples deserve to be called by such a name since there indeed exist continuous maps $s_1:[1,3]\rightarrow \reeb(M)$ and $s_2:[1,3]\rightarrow \reeb(M)$ satisfying (1) $\pi\circ s_1=\pi\circ s_2=\mathrm{id}_{[1,3]}$, and (2) $\im(s_1),\im(s_2)$ lie on $v_1e_1v_3e_3v_5$ and $v_1e_1v_3e_4v_6$ respectively in $\reeb(M)$.\label{item:Reeb2 2}
    \item There exists only one $[1,3]$-full-component; among the two $[1,3]$-components considered in item \ref{item:I-component}, the one containing $v_1$ is the unique $[1,3]$-full-component.\label{item:Reeb2 3}
    \item Via the canonical LC map, each $[1,3]$-section is mapped to the unique $[1,3]$-full-component, which contains the image of the $[1,3]$-section.  This demonstrates why the connected component containing $v_1$ is called a $[1,3]$-full-component of $M$ \emph{with a section (cf. Example \ref{ex:non-isomorphic}).} \label{item:Reeb2 4}
    \end{enumerate}
\end{example}

\subsection{Computing the 0-th level set barcode of a Reeb graph without homology}\label{sec:reeb graph barcode}

In this section we propose a novel method to compute the $0$-th level set barcode of a Reeb graph within the category of sets.

\paragraph{The $0$-th level set barcode of a Reeb graph.} The $0$-th level set barcode of a Reeb graph (Definition \ref{def:0-th levelset barcode}) captures the lifetime of all homological features in the Reeb graph \cite{bauer2014measuring}. 
Whereas measuring the interleaving distance (or equivalently functional distortion distance) between Reeb graphs is not easy \cite[Section 5]{de2016categorified}, \cite{bauer2015strong}, the computation of the bottleneck distance between the $0$-th level set barcodes of Reeb graphs can be efficiently carried out \cite{kerber2017geometry}. Moreover, the bottleneck distance between the $0$-th level set barcodes of Reeb graphs is a tight lower bound for the interleaving distance between Reeb graphs up to multiplicative constant 2 \cite{bauer2014measuring,botnan2018algebraic}.

\begin{definition}[Linearization functor]\label{def:free} Let $\free:\Sets\rightarrow \Vect$ be the \emph{linearization functor} (a.k.a. free functor):  For any set $S$, $\free(S)$ consists of formal linear combination $\sum_{i}a_is_i$, ($a_i\in \F, s_i\in S$) of finite terms of elements in
$S$ over the field $\F$. Also, given a set map $f:S\rightarrow T$ , $\free(f)$ is the linear map from $\free(S)$ to $\free(T)$ obtained by linearly extending $f$. 
\end{definition}

Throughout this section, we identify both of the Grothendieck groups $\A(\vect)$ and $\A(\sets)$ with the integer group $(\Z,+)$ (Remark \ref{rem:counterpart}). Let us recall the $0$-th level set barcode of a Reeb graph \cite{carlsson2009zigzag,botnan2018algebraic}:
\begin{definition}[The 0-th level set barcode of a Reeb graph]\label{def:0-th levelset barcode} Given $M:\ZZ\rightarrow \sets$, consider $\free\circ M:\ZZ\rightarrow \vect$. The $0$-th level set barcode of $M$ refers to $\dgmzz(\free\circ M)$.
\end{definition}

\paragraph{New method for computing the 0-th level set barcode of a Reeb graph.} In order to establish a method to compute the 0-th level set barcode of a Reeb graph \emph{without matrix operations} \cite{carlsson2009zigzag,milosavljevic2011zigzag}, we introduce a new invariant for Reeb graphs. 

Given $M:\ZZ\rightarrow \sets$ and $I\in \inter(\ZZ)$, we denote the number of \emph{full components} of $M|_I$ by $\full(M|_I)$ (Definition \ref{def:full component}). 

\begin{definition}[Full function of a Reeb graph] \label{def:full function}
Given $M:\ZZ\rightarrow \sets$, the \emph{full function} $\Full(M):\inter(\ZZ)\rightarrow \Z_{\geq 0}$ is defined as $I\mapsto \full\left(M|_{I}\right).$
\end{definition}

One can observe that $\Full(M)(\langle b,d \rangle_{\ZZ})$ is equal to the number of connected components of $\pi^{-1}\langle b,d \rangle (\subset \reeb(M))$ \emph{whose images via $\pi$ are $\langle b,d \rangle$}. For instance, consider  $M:\ZZ\rightarrow \sets$ whose corresponding Reeb graph is depicted as in Figure \ref{fig:Curry} (A). $\Full(M)([2,3)_{\ZZ})$ is $2$, since both of the two connected components of $\pi^{-1}[2,3)$ cover $[2,3)$ via $\pi$. On the other hand, $\Full(M)((1,3)_{\ZZ})$ is $1$ since only one of the connected components of $\pi^{-1}(1,3)$ covers the entire $(1,3)$ via $\pi$.

The following proposition is the core observation for establishing a new method to compute the $0$-th level set barcodes of Reeb graphs.

\begin{proposition}\label{thm:full is rank}For any $M:\ZZ\rightarrow \sets$,  \[\Full (M)=\rk_\vect(\free\circ M).\]
\end{proposition}
We prove Proposition \ref{thm:full is rank} at the end of this section.

Let $I\in \inter(\ZZ)$ be finite. Then $\nbd(I)$ consists of two elements, say $\nbd(I):=\{a,b\}$. Hence, we can write the first and second entourage of $I$ as $I^1=\{I\cup\{a\},I\cup\{b\}\}$ and $I^2=\{I\cup\{a,b\}\}$, respectively. By Theorems \ref{thm:interval decomposition0}, \ref{prop:completeness}, and Proposition \ref{thm:full is rank}, we directly have:   

\begin{corollary}\label{cor:new formula for 0th}For any $M:\ZZ\rightarrow \sets$ and any finite $I\in \inter(\ZZ)$, the multiplicity of $I$ in $\dgmzz(\free\circ M)$ is
\[\Full(M)(I)-\Full(M)(I\cup\{a\})-\Full(M)(I\cup\{b\})+\Full(M)(I\cup\{a,b\})\]
which is also equal to $\dgm_{\vect}^\ZZ(\free\circ M)(I)$. 
\end{corollary}
Note that, when $I \in \inter(\ZZ)$ is infinite, the perimeter of $I$ is either 0 or 1 and thus the formula for the multiplicity of $I$ in $\barc(\free\circ M)$ is even simpler than the one given above. 

\begin{example}\label{ex:non-isomorphism thm} Consider  $M:\ZZ\rightarrow \sets$ whose corresponding Reeb graph is depicted as in Figure \ref{fig:Curry} (A). By Corollary \ref{cor:new formula for 0th}, we have that
\begin{align*}
    &\dgm_{\vect}^\ZZ(\free\circ M)([2,3)_{\ZZ})\\
    &=\Full(M)([2,3)_{\ZZ})-\Full(M)((1,3)_{\ZZ})-\Full(M)([2,3]_{\ZZ})+\Full(M)((1,3]_{\ZZ})\\
    &=2-1-1+1\\
    &=1,
\end{align*}
which is the multiplicity of $[2,3)_{\ZZ}$ in $\barc^\ZZ(\free\circ M)$.
By applying the same strategy to the other intervals of $\ZZ$, we have: \begin{equation}\label{eq:dgm of a reeb graph}
    \dgm_{\vect}^\ZZ(\free\circ M)(I)=\begin{cases}
1,&I\in \{[2,3)_{\ZZ},(2,3]_{\ZZ},[1,4]_{\ZZ}\},
\\0,&\mbox{otherwise,}\end{cases}
\end{equation}
and the barcode of $\free\circ M$ consists of the three intervals $[2,3)_{\ZZ},(2,3]_{\ZZ}$ and $[1,4]_{\ZZ}$.
\end{example}

We remark that the Reeb graph from Example \ref{ex:non-isomorphism thm} was introduced in \cite[Figure 12]{adams2015evasion} along the way to address the problem of searching for evasion paths in mobile sensor networks. The authors pointed out  that, without careful homological considerations, one could make a wrong guess about the $0$-th level set barcode of this Reeb graph (see also \cite[Section 10]{curry2013sheaves}). Example \ref{ex:non-isomorphism thm} demonstrates the usefulness of Corollary \ref{cor:new formula for 0th} in this context.
\begin{figure}
    \centering
    \includegraphics[width=\textwidth]{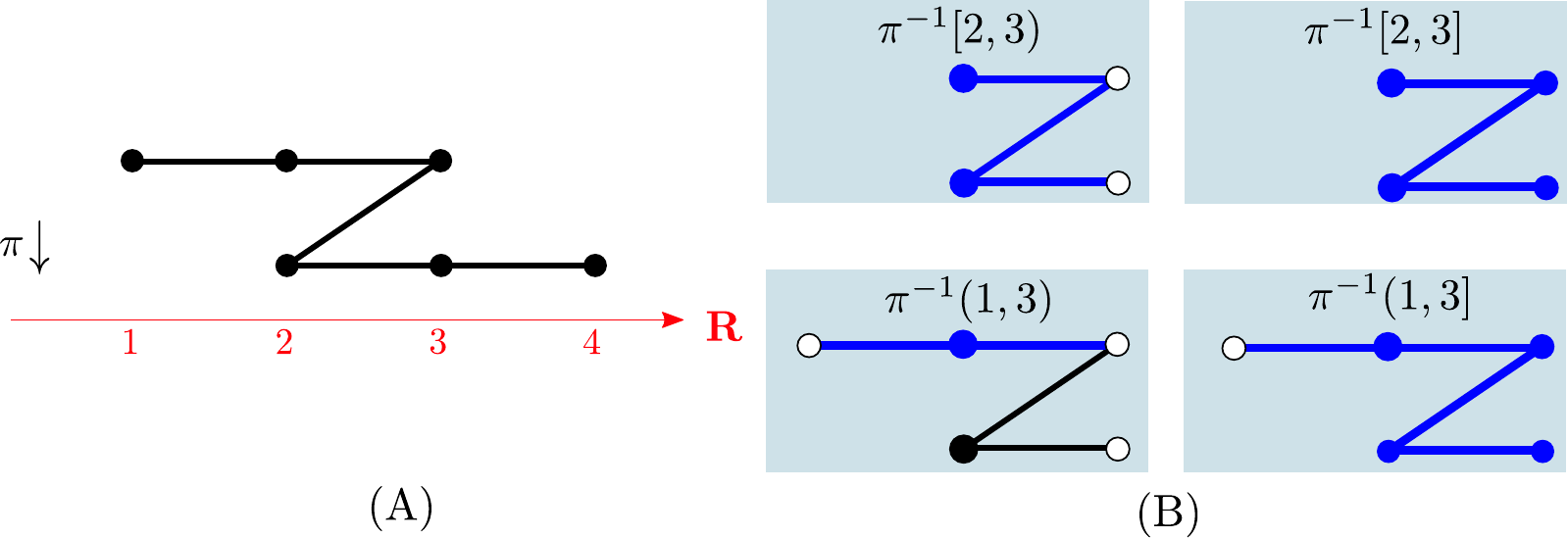}
    \caption{(A) A Reeb graph $(\reeb(M),\pi)$ which corresponds to $M$ in Example \ref{ex:non-isomorphism thm}. (B) Pre-images  $\pi^{-1}(I)$ for 4 different choices of intervals $I\subset \R$. Blue components are the full components over the corresponding intervals.}
    \label{fig:Curry}
\end{figure}

Given $M:\ZZ\rightarrow \sets$, we also have another persistence diagram $\dgm_{\sets}^\ZZ(M)$, which has nothing to do with the linearization functor $\free$. Next we show that $\dgm_{\sets}^\ZZ(M)$ is not necessarily equal to $\dgm^\ZZ_\vect(\free\circ M)$:  

\begin{example}\label{ex:non-isomorphic} Consider  $M:\ZZ\rightarrow \sets$ whose corresponding Reeb graph is depicted as in Figure \ref{fig:Curry} (A). By Definition \ref{def:full components} \ref{item:full components4}, $\rk(M)(\langle b,d\rangle_{\ZZ})$ counts the number of $\langle b,d\rangle$-full components \emph{with a section.} We claim that
\begin{equation}\label{eq:claim}
    \mbox{for $I\in \inter(\ZZ)$,} \hspace{8mm} \rk(M)(I)=\begin{cases} 0,&I \supsetneq [2,3]_{\ZZ}\\
\full(M|_I),&\mbox{otherwise.}
\end{cases}
\end{equation}

Assume that $I:=\langle b,d\rangle_{\ZZ}\in \inter(\ZZ)$ strictly contains $[2,3]_{\ZZ}$. Then, $\pi^{-1}\langle b,d\rangle$ has one connected component and it does not has a section, i.e. there is no map $s:\langle b,d \rangle \rightarrow \reeb(M)$ such that $\pi\circ s = \id_{\langle b,d \rangle}$. On the other hand, if $I$ does not strictly contain $[2,3]_{\ZZ}$, then every $\langle b,d\rangle$-full-connected component has a section. These observations prove the claim in (\ref{eq:claim}). Now, by Definition \ref{def:generalized PD2}, it is not hard to obtain:
\[\dgm_{\sets}^{\ZZ}(M)(I)=\begin{cases} 1,&I=[2,3)_{\ZZ}\ \mbox{or}\ I=(2,3]_{\ZZ}\\ -1,&I=[2,3]_{\ZZ}\\0,&\mbox{otherwise}.
\end{cases}
\]
Now, observe that $\dgm_{\sets}^\ZZ(M)$ is different from $\dgm_{\vect}^\ZZ(\free\circ M)$ in (\ref{eq:dgm of a reeb graph}).
\end{example}

\begin{Remark}
\begin{enumerate}[label=(\roman*)]
    \item The persistence diagram of a $\ZZ$-indexed $\sets$-diagram is not a complete invariant. For example, assume that $N:\ZZ\rightarrow \sets$ is defined as 
    \[N|_{[2,3]_{\ZZ}}=M|_{[2,3]_{\ZZ}}, \hspace{10mm} N_{s}=\emptyset\ \mbox{for $s\notin [2,3]_{\ZZ}$}.\]
Even though $N$ is not isomorphic to $M$ in Example \ref{ex:non-isomorphic}, it is not difficult to check that $\rk(M)=\rk(N)$ and thus $\dgm_{\sets}^{\ZZ}(M)= \dgm_{\sets}^{\ZZ}(N)$.     
    \item For every $M:\ZZ\rightarrow \sets$, $\rk(M)(I)\leq \rk(\free\circ M)(I)$ for all $I=\langle b,d\rangle_{\ZZ}$. This directly follows from the fact that $\rk(M)(I)$ counts $\langle b,d\rangle$-full-components \emph{with a section}, whereas $\rk(\free\circ M)(I)$ counts $\langle b,d\rangle$-full-components (Proposition \ref{thm:full is rank}).
\end{enumerate}
\end{Remark}

In Section \ref{sec:set-stability} we will discuss the stability of $\dgm_{\sets}^\ZZ(M)$ under perturbations of $M:\ZZ\rightarrow \sets$.

\paragraph{Extended persistence and a proof of Proposition \ref{thm:full is rank}.} For proving Proposition \ref{thm:full is rank}, we briefly review the notion of \emph{extended persistence}. See \cite{CEH09} for details.

Let $X$ be a topological space and let $f:X\rightarrow \R$ be a Morse-type function (Definition \ref{def:morse type}) where $S:=\{s_1,s_2,\ldots,s_n\}$ is the collection of critical points of $f$. Let us select a set of indices $t_i$ which satisfy
\[-\infty<t_0<s_1<t_1<s_2<\ldots<t_{n-1}<s_n<t_n<\infty.
\]

For $1\leq i\leq j\leq n$, let $X_i^j:=f^{-1}([t_i,t_j])$. For $k\in\Z_+$, consider the following diagram of absolute and relative homology groups:
\[\begin{tikzcd}
\Hrm_k(X_0^0)\arrow{r}{}&\Hrm_k(X_0^1)\arrow{r}{}&\cdots\arrow{r}{}&\Hrm_k(X_0^{n-1})\arrow{r}{}&\Hrm_k(X_0^n)\arrow{d}{\rotatebox{90}{}}
\\
\Hrm_k(X_0^n,X_0^n)&\arrow{l}{}\Hrm_k(X_0^n,X_1^n)&\arrow{l}{}\cdots&\arrow{l}{}\Hrm_k(X_0^n,X_{n-1}^n)&\arrow{l}{}\Hrm_k(X_0^n,X_n^n),  
\end{tikzcd}\]
where each arrow stands for the map induced by the inclusion.
The collection of pairs arising from this diagram are recorded in the so-called \emph{$k$-th extended persistence diagram}. There are three different types of pairs: ordinary pairs arise between the elements in the top row of the sequence, relative pairs between the elements in the bottom row, and extended pairs span both rows. In particular, in the $k$-th extended persistence diagram,
\begin{enumerate}[label=(\roman*)]
    \item an ordinary pair of $X_0^i$ and $X_0^j$ ($i<j<n$) is recorded as $[s_i,s_{j+1})$,
    \item a relative pair of $(X_{0}^n,X_{i}^n)$ and $(X_{0}^n,X_{j}^n)$ ($i<j<n$) is recorded as $[\bar{s}_{j+1},\bar{s}_{i})$,
    \item an extended pair of $X_0^i$ and $(X_{0}^n,X_{j}^n)$ is recorded as  $[s_i,\bar{s}_j)$.
\end{enumerate}

\begin{proof}[Proof of Proposition \ref{thm:full is rank}] Fix $I:=\langle b,d \rangle_{\ZZ} \in\inter(\ZZ)$. We will prove that 
\begin{equation}\label{eq}
    \full(M)(I)=\rk_{\vect}(\free\circ M) (I).
\end{equation}
By Proposition \ref{rem:implication}, $\rk_{\vect}(\free\circ M)(I)$ is equal to the multiplicity of $I$ in $\barc^I(\free\circ M|_I)$. Let $(\reeb(M|_I),\pi)$ be the Reeb graph corresponding to $M|_I$. By the bijection given in \cite[EP Equivalence Theorem (Table 1)]{carlsson2009zigzag}, the copies of $I=\langle b,d \rangle_{\ZZ}$ in $\dgmzz(\free\circ M|_I)$ one-to-one correspond to the copies of $[b,\bar{d})$ in the $0$-th extended persistence diagram of $(\reeb(M|_I),\pi)$.  Also, the copies of $[b,\bar{d})$ in the $0$-th extended persistence diagram of $(\reeb(M|_I),\pi)$ one-to-one correspond to the connected components of $\reeb(M|_I)$ whose image via $\pi$ is $[b,d]\subset \R$ \cite[p.83]{CEH09}. These connected components are exactly the full components of $M|_I$ in the sense of Definition \ref{def:full component}. Therefore, we have the equality in (\ref{eq}). 
\end{proof}

\subsection{The stability of $\sets$-persistence diagrams for merge trees and untwisted Reeb graphs}\label{sec:set-stability}

Patel's semicontinuity theorem \cite{patel2018generalized} states that the so-called \emph{type $\A$ persistence diagram} of $F:\R\rightarrow \C$ is stable to all \emph{sufficiently small} perturbations of $F$. In this section we promote this theorem to a complete continuity theorem when $\C=\sets$. Also, in turn we establish a stability result for $\sets$-persistence diagrams of a certain class of Reeb graphs.

Again, throughout this section, we identify both of the Grothendieck groups $\A(\vect)$ and $\A(\sets)$ with the integer group $(\Z,+)$ (Remark \ref{rem:counterpart}).

\paragraph{Bottleneck stability for merge trees.} One can encode all the combinatorial information of a \emph{merge tree} \cite{morozov2013interleaving} into a $\Z$-indexed $\sets$-diagram:

\begin{definition}A functor $F:\Z\rightarrow \sets$ is said to be a \emph{(discrete) merge tree}.\footnote{A \emph{constructible} persistence
module $F:\R\rightarrow \sets$ (Definition \ref{def:constructible} in Appendix) is often said to be a \emph{merge tree} \cite{morozov2013interleaving}. In order to compute the Patel's persistence diagram of $F$, it suffices to consider a certain re-indexed diagram $D(F):\Z\rightarrow \sets$ of $F$. Such a re-indexing method is described in Appendix \ref{sec:patel} (the paragraph \emph{Re-indexing a constructible persistence module by $\Z$}).}
\end{definition}

Recall from Section \ref{sec:reeb graph barcode} that, given an $M:\ZZ\rightarrow \sets$, the two persistence diagrams $\dgm^{\ZZ}_{\sets}(M)$ and $\dgm^{\ZZ}_{\vect}(\free\circ M)$ do not coincide in general. That is not the case for merge trees:

\begin{proposition}\label{prop:isomorphism-merge tree} For any $F:\Z\rightarrow \sets$, we have $\rk(F)=\rk(\free\circ F)$ and therefore, $$\dgm_{\sets}^{\Z}(F)=\dgm_{\vect}^{\Z}(\free\circ F).$$
\end{proposition}

Before proving Proposition \ref{prop:isomorphism-merge tree}, note the following: given a set map $f:A\rightarrow B$, it directly follows from the definition of the linearization functor $\free$ that the cardinality of $\im(f)$ is equal to the rank of $\free(f):\free(A)\rightarrow \free(B)$. 

\begin{proof}[Proof of Proposition \ref{prop:isomorphism-merge tree}] Let $I:=[a,b]\in \inter(\Z)$. Then,
\[\rk(F)(I)=\abs{\im (\varphi_F(a\leq b))}=\rank\left({\free(\varphi_F(a\leq b))}\right)=\rk(\free\circ F)(I).
\]
Since, $\dgm_{\sets}^{\Z}(F)$ (resp. $\dgm_{\vect}^{\Z}(\free\circ F)$) is the M\"obius inversion of $\rk(F)$ (resp. $\rk(\free\circ F)$) over the poset $(\inter(\Z),\supset)=(\con(\Z),\supset)$ (Definition \ref{def:generalized PD2}), we also have $\dgm_{\sets}^{\Z}(F)=\dgm_{\vect}^{\Z}(\free\circ F)$.
\end{proof}

We remark that the stability of $\dgm_{\vect}^{\Z}(\free\circ F)$ under perturbations of $F$ was established in \cite{morozov2013interleaving}, whereas the semicontinuity of $\dgm_{\sets}^{\Z}(F)$ was established in \cite[Theorem 8.1]{patel2018generalized}\footnote{This theorem states that, when $\C$ is essentially small symmetric monoidal category with images, the persistence diagram of $F:\R\rightarrow \C$ is stable to all \emph{sufficiently small} perturbations of $F$.} (we remark that both of those works use $\R$ as the indexing poset). Now, by virtue of Proposition \ref{prop:isomorphism-merge tree}, \cite[Theorem 8.1]{patel2018generalized} can be improved to a complete stability theorem: Let $\bott$ be the bottleneck distance, and let $\dint^{\C}$ be the interleaving distance between functors $\Z\rightarrow \C$ (Definitions \ref{def:interleaving_general} and \ref{def:bottleneck distance} in Appendix). We have:

\begin{corollary}\label{cor:promotion}For $F,G:\Z\rightarrow \sets$,
\[\bott\left(\dgm_{\sets}^\Z(F),\dgm_{\sets}^\Z(G)\right)\leq \dint^\sets(F,G).
\]
\end{corollary}
Though this corollary directly follows from Proposition \ref{prop:isomorphism-merge tree} and arguments similar to those is \cite{morozov2013interleaving}, we provide another concise proof: 
\begin{proof}We have:
\begin{align*}
    \bott\left(\dgm_{\sets}^{\Z}(F),\dgm_{\sets}^{\Z}(G)\right)&=\bott\left(\dgm_{\vect}^{\Z}(\free\circ F),\dgm_{\vect}^{\Z}(\free\circ G)\right)&\mbox{by Proposition \ref{prop:isomorphism-merge tree}}\\
    &= \dint^\vect(\free\circ F,\free\circ G)&\mbox{see below}\\
    &\leq \dint^\sets(F,G)&\mbox{by functoriality of $\free$.}
    \end{align*}
The second equality follows from the celebrated isometry theorem \cite{chazal2016structure,lesnick}.    
\end{proof}
We remark that Corollary \ref{cor:promotion} is true even for constructible functors $F,G:\R\rightarrow\sets$ (Definition \ref{def:constructible} in Appendix), which can be proved via trivial re-indexing of $F,G$ by $\Z$.

\paragraph{Decomposition of a Reeb graph and untwisted Reeb graphs.} We define \emph{untwisted Reeb graphs}, a generalization of merge trees, and extend Corollary \ref{cor:promotion} to untwisted Reeb graphs.

A Reeb graph $M:\ZZ\rightarrow \sets$ is said to be \emph{empty} if $M_s=\emptyset$ for all $s\in \ZZ$. Let $N,N':\ZZ\rightarrow \sets$. The \emph{disjoint union} $N\coprod N':\ZZ\rightarrow \sets$ is defined as follows: for all $s\in \ZZ$, $(N\coprod N')_{s}:=N_s\sqcup N_s'$ and for all $s\leq t$ in $\ZZ$, $\varphi_{N\coprod N'}(s,t):(N\coprod N')_s \rightarrow (N\coprod N')_t$ is defined as $\varphi_{N}(s,t)\sqcup \varphi_{N'}(s,t)$. We say that a nonempty $M:\ZZ\rightarrow \sets$ is \textit{decomposable} if $M$ is isomorphic to $N \coprod N'$ for some nonempty  $N,N':\ZZ\rightarrow \sets$. Otherwise, we say that $M$ is \textit{indecomposable}. We have similar definitions even when the indexing poset $\ZZ$ is replaced by any $I\in\inter(\ZZ)$. In what follows, we will see that the decomposition of a nonempty $M:\ZZ\rightarrow \sets$ into indecomposables parallels the topological decomposition of $\reeb(M)$ (Definition \ref{def:realization}) into path-connected components.

We can explicitly decompose $M$ as follows: Let $J:=\varinjlim M$, which is a nonempty set (not necessarily finite). According to the canonical construction in equation (\ref{eq:set colimit construction}), each $j\in J$ is a subset of the disjoint union $\coprod_{k\in \ZZ} M_{k}$. Note that $M\cong \coprod_{j\in J}N^j$, where each $N^j:\ZZ\rightarrow \sets$ is defined as $N^j_k=\{m\in M_k: m\in j\}$ for $k\in \ZZ$ and the internal morphisms of $N$ are the canonical restrictions of the internal morphisms of $M$. Now we directly have:

\begin{proposition} Any nonempty $M:\ZZ\rightarrow \sets$ decomposes into a disjoint union of indecomposables $\ZZ\rightarrow \sets$, and the decomposition is unique up to a permutation of summands.
\end{proposition}

\begin{Remark}\label{rem:Reeb graph decomposition}
\begin{enumerate}[label=(\roman*)]
\item Recall that a nonempty $M:\ZZ\rightarrow \sets$ gives rise to the Reeb graph $(\reeb(M),\pi)$ (Definition \ref{def:realization}). Given an indecomposable decomposition $M\cong \coprod_{j\in J} N^j$, the path-connected components $(\reeb(M),\pi)$ are precisely $\reeb(N^j)$, $j\in J$ equipped with the restrictions of $\pi$.  \label{item:Reeb graph decomposition1}

\item Given any nonempty $M:\ZZ\rightarrow \sets$, the following are equivalent: (a) a nonempty $M$ is indecomposable, (b) $\reeb(M)$ (Definition \ref{def:realization}) contains only one path-connected component, (c) $\varinjlim M$ is a singleton.  \label{item:Reeb graph decomposition2}
\end{enumerate}
\end{Remark}

Let $M:\ZZ\rightarrow \sets$ be nonempty indecomposable. Then, $\varinjlim M$ is a singleton by Remark \ref{rem:Reeb graph decomposition} \ref{item:Reeb graph decomposition2}. By $\supp(M)$, we denote the support of the unique element of $\varinjlim M$ (Definition \ref{def:full component}).

\begin{definition}[Untwisted Reeb graphs]\label{def:untwisted} 
A nonempty $M:\ZZ\rightarrow \sets$ is said to be \emph{untwisted} if the following holds: For $I\in \inter(\ZZ)$ such that $M|_{I}$ is nonempty, if $M|_I\cong \coprod_{j\in J} N^j$ for some indexing set $J$ and nonempty indecomposables $N^j:I\rightarrow \sets$, then $\varprojlim N^j|_{\supp(N^j)}\neq \emptyset$ for each $j\in J$. 
\end{definition}

Examples of untwisted Reeb graphs include $M$ in Example \ref{ex:Reeb} and merge trees\footnote{Every $\Z$-indexed $\sets$-diagram $F$ can be converted into a $\ZZ$-indexed diagram $D(F)$ which contains the same combinatorial information as $F$ (refer to the paragraph \emph{Re-indexing $\Z$-indexed diagram by $\ZZ$} in Appendix \ref{sec:patel}). In this respect, every merge tree can be viewed as a Reeb graph.}, while non-examples include the Reeb graph of Figure \ref{fig:Curry} (A). Let us characterize untwisted Reeb graphs:

\begin{proposition}\label{prop:characterization of untwisted reeb graphs}
Let $M:\ZZ\rightarrow \sets$ be nonempty and let $(\reeb(M),\pi)$ be the corresponding Reeb graph (Definition \ref{def:realization}). The following are equivalent:
\begin{enumerate}[label=(\roman*)]
    \item $M$ is untwisted.\label{item:characterization1}
    \item $\rk(M)=\rk(\free\circ M)$  (and thus, $\dgm_{\sets}^\ZZ(M)=\dgm_{\vect}^\ZZ(\free\circ M)$).\label{item:characterization2}
    \item  For each real interval $\langle b,d \rangle$ such that  $\pi^{-1}(\langle b,d\rangle)$ is nonempty, the restriction of $\pi$ to each connected component of $\pi^{-1}(\langle b,d\rangle)$ has a section. \label{item:characterization3}

\end{enumerate}

\end{proposition}

\begin{proof}
We will utilize geometric insight of Example \ref{ex:Reeb2} and Remark \ref{rem:Reeb graph decomposition} throughout the proof.

\ref{item:characterization1}$\Rightarrow$\ref{item:characterization3}: It suffices to consider an interval $\langle b,d \rangle\subset \R$ such that $b,d\in\Z\cup\{-\infty,\infty\}$ and $\pi^{-1}(\langle b,d \rangle)$ is nonempty. Let $I:=\langle b,d\rangle_{\ZZ}$ and consider the restricted diagram $M|_I$. Note that the assumption $\pi^{-1}(\langle b,d \rangle)\neq \emptyset$ implies that $M|_I$ is nonempty. Now assume that $M|_I$ has an indecomposable decomposition $M|_I\cong \coprod_{j\in J} N^{j}$.  Then $\reeb(N^j)$ are exactly the connected components of $\reeb(M|_I)$ (Remark \ref{rem:Reeb graph decomposition} \ref{item:Reeb graph decomposition1}). Also, there exists a canonical bijection from $\varprojlim N^j|_{\supp(N^j)}$ to the sections of the restriction $\pi|_{\reeb(N^j)}$ (cf. Example \ref{ex:Reeb2} \ref{item:Reeb2 2}). For each $j\in J$, since $\varprojlim N^j|_{\supp(N^j)}\neq \emptyset$, $\pi|_{\reeb(N^j)}$ has at least one section.

\ref{item:characterization3}$\Rightarrow$\ref{item:characterization2}: 
Fix $I=\langle b,d \rangle_{\ZZ}$. By assumption, every $\langle b,d \rangle$-full-component has a section. Invoking that $\rk(\free\circ M)(I)$ is the number of $\langle b,d \rangle$-full-components (Proposition \ref{thm:full is rank}), whereas $\rk(M)(I)$ stands for the number of $\langle b,d \rangle$-full-components \emph{with a section} (Definition \ref{def:full components} \ref{item:full components4} and Example \ref{ex:Reeb2} \ref{item:Reeb2 4}), we have $\rk(M)(I)=\rk(\free\circ M)(I)$. 

\ref{item:characterization2}$\Rightarrow$\ref{item:characterization1}:  Fix any $I\in \inter(\ZZ)$ such that $M|_{I}$ is nonempty and assume that $M|_I\cong \coprod_{j\in J} N^j$ for some indexing set $J$ and nonempty indecomposables $N^j:I\rightarrow \sets$. 

Fix $j\in J$ and let $K:=\supp(N^j)=\langle b,d \rangle_{\ZZ}$. By assumption, the number $\rk(M)(K)$ of $\langle b,d \rangle$-full-components  of $\reeb(M)$ is equal to the number $\rk(\free\circ M)(K)$ of $\langle b,d \rangle$-full-components of $\reeb(M)$ \emph{which have a section}. 
Therefore, every $\langle b,d \rangle$-full-component has a section. By the choice of $K$, $\reeb(N^j)$ is a $\langle b,d \rangle$-full-component and thus $\reeb(N^j)$ has a section. This means that $\varprojlim N^j|_{K} \neq \emptyset$, as desired.
\end{proof}

\paragraph{Bottleneck stability for untwisted Reeb graphs.} We extend the stability result in Corollary \ref{cor:promotion} to a class of \emph{untwisted Reeb graphs}.

An \emph{interleaving} distance $\dint^{\C}$ between zigzag modules $\ZZ\rightarrow \C$ is defined when $\C$ is cocomplete (Definition \ref{def:interleaving by botnan} in Appendix). The following theorem can be proved in the same way as Corollary \ref{cor:promotion} except utilizing the algebraic stability of zigzag modules (Theorem \ref{thm:bottleneck stability} in Appendix) in lieu of the isometry theorem.

\begin{theorem}\label{thm:stability for sets}For any untwisted Reeb graphs $M,N:\ZZ\rightarrow \sets$,
\begin{equation}\label{eq:stability for reeb graphs}
    \bott\left(\dgm_\sets^\ZZ(M),\dgm_\sets^\ZZ(N)\right)\leq  2\cdot\dint^\Sets(M,N).
\end{equation}
\end{theorem}

\begin{proof}We have:
\begin{align*}
    \bott\left(\dgm_{\sets}^{\ZZ}(M),\dgm_{\sets}^{\ZZ}(N)\right)&=\bott\left(\dgm_{\vect}^{\ZZ}(\free\circ M),\dgm_{\vect}^{\ZZ}(\free\circ N)\right)&\mbox{by Proposition \ref{prop:characterization of untwisted reeb graphs}}\\
    &\leq 2\cdot \dint^\Vect(\free\circ M,\free\circ N)&\mbox{Theorem \ref{thm:bottleneck stability}}\\
    &\leq 2\cdot\dint^\Sets(M,N)&\mbox{by functoriality of $\free$.}
    \end{align*}
\end{proof}

We remark that the inequality in (\ref{eq:stability for reeb graphs}) is tight:  $\dgm_\sets^\ZZ(M)$ and $\dgm_\sets^\ZZ(N)$ are equal to $\dgm_\vect^\ZZ(\free\circ M)$ and $\dgm_\vect^\ZZ(\free\circ N)$ respectively and the tightness of the following inequality is known \cite{botnan2018algebraic}:
\[ \bott\left( \dgm_\vect^\ZZ(\free\circ M), \dgm_\vect^\ZZ(\free\circ N)\right)\leq  2\cdot\dint^\sets(M,N).
\]
We provide a concrete example of which the both sides in (\ref{eq:stability for reeb graphs}) coincide:
\begin{example}[Tightness]\label{ex:tight}Let $M:\ZZ\rightarrow \sets$ be defined as
\[M_{(0,0)}=\{v_1\},\ M_{(1,1)}=\{v_2\},\ \ M_{(1,0)}=\{e_1,e_2\}\]
and any other $M_{(i,j)}$ is the empty set. The maps $\varphi_M((1,0),(0,0)),\varphi_M((1,0),(1,1))$ are the unique surjections See Figure \ref{fig:tightness} (A).  One can check that $\dgm_{\sets}^\ZZ(M)(I)= 1$ if $I=[1,2]_{\ZZ}$ or $I=(1,2)_{\ZZ}$, and $\dgm_{\sets}^\ZZ(M)(I)= 0$ otherwise. 

On the other hand, define $N:\ZZ\rightarrow \sets$ as 
\[N_{(0,0)}=\{v_1\},\ N_{(1,1)}=\{v_2\},\ \ N_{(1,0)}=\{e_1\}\] and any other $N_{(i,j)}$ is the empty set.  The maps $\varphi_N((1,0),(0,0)),\varphi_N((1,0),(1,1))$ are the unique bijections. See Figure \ref{fig:tightness} (B).  One can check that $\dgm_{\sets}^\ZZ(N)(I)= 1$ if $I=[1,2]_{\ZZ}$ and $\dgm_{\sets}^\ZZ(N)(I)= 0$ otherwise. Checking that \[\bott\left(\dgmpd_\sets(M),\dgmpd_\sets(N)\right)=1/2,\hspace{2mm}\mbox{and}\ \dint^{\sets}(M,N)=1/4,\]
demonstrates the tightness of the inequality in (\ref{eq:stability for reeb graphs}).
\end{example}

\begin{figure}
    \centering
    \includegraphics[width=0.7\textwidth]{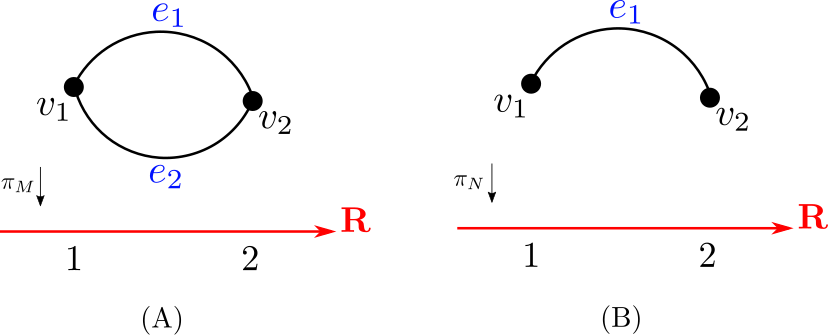}
    \caption{Illustration for Example \ref{ex:tight}. (A) The Reeb graph corresponding to $M$. (B)  The Reeb graph corresponding to $N$. Observe that for every interval $I\in\R$, the preimages $\pi_M^{-1}(I)$ and $\pi_N^{-1}(I)$ consist solely of connected components which allow a section. This demonstrate that $M$ and $N$ are untwisted.}
    \label{fig:tightness}
\end{figure}

\section{Discussion}\label{sec:discussion}

We have extended the notions of rank invariant and generalized persistence diagram due to Patel \cite{patel2018generalized} to the setting of generalized persistence modules $\Pb\rightarrow \C$, when $\Pb$ and $\C$ satisfy mild assumptions.
The rank invariant and persistence diagram defined in this paper generalize those in \cite{botnan2018algebraic,carlsson2009zigzag,carlsson2009theory,cohen2007stability,patel2018generalized,ZC05}. In particular, our construction of the rank invariant yields a complete invariant for interval decomposable persistence modules $\Pb\rightarrow \vect$ (which include the case of zigzag modules). Along the way, the $0$-th level set barcode of a Reeb graph has been interpreted with a novel viewpoint. This leads to the promotion of a Patel's theorem and opens up the possibility of the existence of another efficient algorithm for computing the $0$-th level set barcode of a Reeb graph or a Morse-type function.

Many questions remain unanswered: \begin{itemize}
    \item[(1)] To what extent can we generalize  stability results about persistence diagrams/barcodes beyond those results in Section \ref{sec:set-stability} and those in  \cite{botnan2018algebraic,CCG09,cohen2007stability,mccleary2018bottleneck,patel2018generalized}? In particular, for arbitrary $M:\ZZ\rightarrow \sets$: how to address the stability of $\dgm_{\sets}^{\ZZ}(M)$ given that $\dgm_{\sets}^{\ZZ}(M)$ could take negative values?\footnote{We saw such a case in Example \ref{ex:non-isomorphic}. Non-positive persistence diagrams also appear in \cite{betthauser2019graded}.} See Theorem \ref{thm:stability of gen rank} for
the stability of the generalized rank invariant (Definition \ref{def:fully generalized rank invariant}) with respect to the interleaving
distance.
    
    \item[(2)] For persistence modules $\Pb\rightarrow \vect$ which are \emph{not} interval decomposable, how faithful is the rank invariant? 

\item[(3)] We wonder whether interval modules can be characterized in terms of persistence diagrams as follows. Given an \emph{indecomposable} $F:\Pb\rightarrow \vect$, if $\dgm^{\Pb}(F)$ vanishes on $\con(\Pb)\setminus \inter(\Pb)$ and has non-negative values on $\inter(\Pb)$: is $F$ an interval module? This question is motivated from Example \ref{ex:Alex}. 

\item[(4)] In order to compute level set barcodes of Morse-type functions, algorithms in \cite{carlsson2009zigzag,milosavljevic2011zigzag} make use of a sequence of operations on matrices. Utilizing the results in Section \ref{sec:reeb graph barcode}, can we establish more efficient algorithms for specifically computing  $0$-th level set barcodes?  

\item[(5)] An $\R^d$-indexed persistence module $M$ can often be encoded  as a $\Pb$-indexed module $F_M$ for some connected finite poset $\Pb$; see \cite[Section 1.3]{miller2020homological}. Small perturbations to $M$ in the interleaving distance (Definition \ref{def:interleaving_general}) are expected to result in small variations of $\dgm^{\Pb}(F_M)$. How can we quantify the \emph{stability} of the assignment $M\mapsto \dgm^{\Pb}(F_M)$?
Answering  this question will enable us to utilize $\dgm^{\Pb}(F_M)$ in statistical studies of multiparameter persistence modules (e.g. efficient clustering methods for multiparameter persistence modules).
\end{itemize}



\appendix

\section{Limits and Colimits}
We recall the notions of limit and colimit \cite{mac2013categories}. Throughout this section $I$ will stand for a small category.
\begin{definition}[Cone]\label{def:cone} Let $F:I\rightarrow \C$ be a functor. A \emph{cone} over $F$ is a pair $\left(L,(\pi_x)_{x\in \ob(I)}\right)$ consisting of an object $L$ in $\C$ and a collection $(\pi_x)_{x\in \ob(I)}$ of morphisms $\pi_x:L \rightarrow F(x)$ that commute with the arrows in the diagram of $F$, i.e. if $g:x\rightarrow y$ is a morphism in $I$, then $\pi_y= F(g)\circ \pi_x$ in $\C$, i.e. the diagram below commutes.
\end{definition}
\[\begin{tikzcd}F(x)\arrow{rr}{F(g)}&&F(y)\\
& L \arrow{lu}{\pi_x} \arrow{ru}[swap]{\pi_y}\end{tikzcd}\]

In Definition \ref{def:cone}, the cone $\left(L,(\pi_x)_{x\in \ob(I)}\right)$ over $F$ will sometimes be denoted simply by $L$, suppressing the collection $(\pi_x)_{x\in \ob(I)}$ of morphisms if no confusion can arise. A limit of a diagram $F:I\rightarrow \C$ is a terminal object in the collection of all cones:

\begin{definition}[Limit]\label{def:limit} Let $F:I\rightarrow \C$ be a functor. A \emph{limit} of $F$ is a cone over $F$, denoted by $\left(\varprojlim F,\ (\pi_x)_{x\in \ob(I)} \right)$ or simply $\varprojlim F$, with the following \emph{terminal} property: If there is another cone $\left(L',(\pi'_x)_{x\in \ob(I)} \right)$ of $F$, then there is a \emph{unique} morphism $u:L'\rightarrow \varprojlim F$ such that $\pi_x'=\pi_x\circ u$ for all $x\in \ob(I)$.
\end{definition}

\begin{Remark}It is possible that a diagram does not have a limit at all. However, if a diagram does have a limit then the terminal property of the limit guarantees its uniqueness up to isomorphism. For this reason, we will sometimes refer to a limit as \emph{the} limit of a diagram. 
\end{Remark}

Cocones and colimits are defined in a dual manner:
\begin{definition}[Cocone]\label{def:cocone} Let $F:I\rightarrow \C$ be a functor. A \emph{cocone} over $F$ is a pair $\left(C,(i_x)_{x\in \ob(I)}\right)$ consisting of an object $C$ in $\C$ and a collection $(i_x)_{x\in \ob(I)}$ of morphisms $i_x:F(x)\rightarrow C$ that commute with the arrows in the diagram of $F$, i.e. if $g:x\rightarrow y$ is a morphism in $I$, then $i_x= i_y\circ F(g)$ in $\C$, i.e. the diagram below commutes.
\end{definition}
\[\begin{tikzcd}
&C\\
F(x)\arrow{ru}{i_x}\arrow{rr}[swap]{F(g)}&&F(y)\arrow{lu}[swap]{i_y}
\end{tikzcd}\]

In Definition \ref{def:cocone}, a cocone $\left(C,(i_x)_{x\in \ob(I)}\right)$ over $F$ will sometimes be denoted simply by $C$, suppressing the collection $(i_x)_{x\in \ob(I)}$ of morphisms. A colimit of a diagram $F:I\rightarrow \C$ is an initial object in the collection of cocones over $F$:

\begin{definition}[Colimit]\label{def:colimit}Let $F:I\rightarrow \C$ be a functor. A \emph{colimit} of $F$ is a cocone, denoted by $\left(\varinjlim F,\ (i_x)_{x\in \ob(I)}\right)$ or simply $\varinjlim F$, with the following \emph{initial} property: If there is another cocone $\left(C', (i'_x)_{x\in \ob(I)}\right)$ of $F$, then there is a \emph{unique} morphism $u:\varinjlim F\rightarrow C'$  such that $i'_x=u\circ i_x$ for all $x\in \ob(I)$.
\end{definition}

\begin{Remark}It is possible that a diagram does not have a colimit at all. However, if a diagram does have a colimit then the initial property of the colimit guarantees its uniqueness up to isomorphism. For this reason, we will sometimes refer to a colimit as \emph{the} colimit of a diagram. 
\end{Remark}

\begin{Remark}[Restriction of an indexing poset]\label{rem:restriction}Let $\Pb$ be any poset and let $\Q$ be a subposet of $\Pb$. In categorical language, $\Q$ is a full subcategory of $\Pb$. Let $F:\Pb \rightarrow \C$ be a functor.
\begin{enumerate}[label=(\roman*)]
    \item  Assume that the limit of the restriction $F|_{\Q}$ exists. For any cone $\left(L', (\pi_p')_{p\in \Pb}\right)$ over $F$, its restriction $\left(L', (\pi_p')_{p\in \Q}\right)$ is a cone over the restriction $F|_\Q:\Q\rightarrow \C$. Therefore, by the terminal property of the limit  $\left(\varprojlim F|_{\Q}, (\pi_q)_{q\in \Q}\right)$, there exists the unique morphism $u:L' \rightarrow \varprojlim F|_{\Q}$ such that  $\pi_q'=\pi_q\circ u$ for all $q\in \Q$.
    \item  Assume that the colimit of the restriction $F|_{\Q}$ exists. For any cocone $\left(C', (i_p')_{p\in \Pb}\right)$ over $F$, its restriction $\left(C', (i_p')_{p\in \Q}\right)$ is a cocone over the restriction $F|_\Q:\Q\rightarrow \C'$. Therefore, by the initial property of $\varinjlim F|_{\Q}$, there exists the unique morphism $u:\varinjlim F|_{\Q} \rightarrow C'$ such that  such that $i'_q=u\circ i_q$ for all $q\in \Q$.
\end{enumerate}
\end{Remark}

\section{The rank invariant of a standard  persistence module}\label{sec:persistence modules}

In this section we review some important (standard) results about the rank invariant of one-dimensional or multidimensional persistence modules. Recall the poset $\U$ from Definition \ref{def:posets}.

\begin{definition}[Rank invariant of a persistence module]\label{def:rank invariant for a persistence module}Let $F:\R\rightarrow \vect$ be any persistence module. The \emph{rank invariant} of $F$ is defined as the map $\rk(F):\U\rightarrow \Z_+$ which sends each $\bu=(u_1,u_2)\in\U$ to $\mathrm{rank}\left(F(u_1\leq u_2)\right)$.
\end{definition}

\begin{Remark}[Category theoretical interpretation of the rank invariant]\label{rem:rank invariant of a persistence module}  Let $F:\R\rightarrow \vect$ be a persistence module. For any $\bu=(u_1,u_2)\in \U$, it is not difficult to check that \[\left(F_{u_1},\left(\varphi_F(u_1,t)\right)_{t\in[u_1,u_2]}\right)\ \mbox{and}\  \left(F_{u_2},\left(\varphi_F(t,u_2)\right)_{t\in[u_1,u_2]}\right)\] are a limit and a colimit of $F|_{[u_1,u_2]}$, respectively. The canonical LC map from $F_{u_1}$ to $F_{u_2}$ is definitely $\varphi_F(u_1,u_2)$, which is identical to $\varphi_F(t,u_2)\circ\varphi_F(u_1,t)$, for any $t\in [u_1,u_2]$. Therefore, $\rk(F)(\bu)$ can be regarded as the rank of the canonical LC map of $F|_{[u_1,u_2]}$.
\end{Remark}

\begin{Remark}\label{rem:meaning1} In Definition \ref{def:rank invariant for a persistence module}, $\rk(F)(\bu)$ for $\bu=(u_1,u_2)$ counts all the persistence features of the persistence module $F$ which are born before or at $u_1$ and die after $u_2$. Also, when $u_1=u_2$, $\rk(M)(\bu)$ is the dimension of $F_{u_1}$.\end{Remark}

\begin{Remark}[Rank invariant is order-reversing]\label{rem:rank invariant decreases}In Definition \ref{def:rank invariant for a persistence module}, for any pair $\bu=(u_1,u_2)\leq\bu'=\left(u_1',u_2'\right)$ in $\U$, since  \[\varphi_F\left(u_1', u_2'\right)=\varphi_F\left(u_2, u_2'\right)\circ \varphi_F(u_1, u_2)\circ \varphi_F\left(u_1', u_1\right),\] it holds that $\rk(F)\left(\bu'\right)\leq \rk(F)(\bu).$ Therefore, the map $\rk(F):\U\rightarrow \Z_+$ is an order-reversing map. This result generalizes to \cite[Proposition 4.4]{mccleary2018bottleneck}. Also, see \cite[Proposition 3.7]{puuska2017erosion}.  
\end{Remark}

\begin{theorem}[Completeness of the rank invariant for one-dimensional modules \cite{carlsson2009theory}]\label{thm:persistence module}
 The rank invariant defined in Definition \ref{def:rank invariant for a persistence module} is a complete invariant for one-dimensional persistence modules, i.e. if there are two \emph{constructible} persistence modules $F,G:\R\rightarrow \vect$ such that $\rk(F)=\rk(G)$, then $F$ and $G$ are isomorphic (see Definition \ref{def:constructible} for the meaning of \emph{constructible}).
\end{theorem}

The rank invariant can also be defined for multidimensional modules $F:\R^n \rightarrow \vect$, $n>1$ : For any pair $\ba\leq \bb$ in $\R^n$, let $\rk(F)(\ba,\bb):=\mathrm{rank}\left(\varphi_F(\ba, \bb)\right)$. This defines a function from the set $\{(\ba,\bb)\in \R^n\times \R^n: \ba\leq \bb \}$ to $\Z_+$. However, the map $\rk(F)$ is not a complete invariant for multidimensional modules, i.e. for any $n>1$, there exists a pair of persistence modules $F,G:\R^n\rightarrow \vect$ that are not isomorphic but $\rk(M)=\rk(N)$ \cite{carlsson2009theory}.


\section{Comparison with Ville Puuska's rank invariant}\label{sec:puuska}

In \cite{puuska2017erosion}, Ville Puuska considers the set
\[\Dgm_\Pb:=\{(a,b)\in \Pb \times \Pb: a<b\}\] and defines
the rank invariant of a functor $F:\Pb\rightarrow \C$ as the map $dF:\Dgm_{\Pb}\rightarrow \C$ sending $(a,b)$ to $\im \left(\varphi_F(a,b)\right)\in \ob(\C)$. Even though this definition is a straightforward generalization of the rank invariant of \cite{carlsson2009theory}, when $\Pb=\ZZ$ and $\C=\vect$, this definition is not anywhere near a complete invariant of $F:\Pb\rightarrow \C$. Namely, there exists a pair of zigzag modules $M,N$ such that $\dint(M,N)=+\infty$ (Definition \ref{def:interleaving by botnan}) whereas $dM\cong dN$:

\begin{example}\label{ex:criticism} Consider the two zigzag modules $M,N:\ZZ\rightarrow \vect$ defined as follows: $M:=I^{(-\infty,\infty)_\ZZ}$ and $N$ is defined as
\begin{align*}
    &N_{(i,i)}=\F, \hspace{10mm}N_{(i+1,i)}=\F^2,\\
    &\varphi_N((i,i-1),(i,i))=\pi_1,\\&\varphi_N((i+1,i),(i,i))=\pi_2,
\end{align*}
where $\pi_1,\pi_2:\F^2\rightarrow \F$ are the canonical projections to the first and the second coordinate, respectively. Note that $dM(a,b)\cong dN(a,b)\cong \F$ for all $(a,b)\in \Dgm_\Pb$. However, it is not difficult to check that $\dgmzz(M)=\lmulti(-\infty,\infty)_{\ZZ}\rmulti$ and $\dgmzz(N)=\lmulti (i,i+2)_{\ZZ}:i\in \Z \rmulti$. This implies that $\bott\left(\dgmzz(M),\dgmzz(N)\right)=+\infty$, and in turn $\dint(M,N)=+\infty$ by Theorem \ref{thm:bottleneck stability}. 

\end{example}

\section{Interleaving distance and existing stability theorems}\label{sec:the proof of isomorphism thm}

\subsection{Interleaving distance}\label{sec:interleaving distance}

We review the \emph{interleaving distance} between $\R^d$ (or $\Z^d$)-indexed functors and between $\ZZ$-indexed functors \cite{botnan2018algebraic,CCG09,lesnick}.

\paragraph{Natural transformations.} We recall the notion of \emph{natural transformations} from category theory \cite{mac2013categories}: Let $\mathcal{C}$ and $\mathcal{D}$ be any categories and let $F,G:\mathcal{C}\rightarrow \mathcal{D}$ be any two functors. A natural transformation $\psi:F\Rightarrow G$ is a collection of morphisms $\psi_c: F_c\rightarrow G_c$ in $\mathcal{D}$  for all objects $c\in\mathcal{C}$ such that for any morphism $f:c\rightarrow c'$ in $\mathcal{C}$, the following diagram commutes:

\begin{center}
\begin{tikzcd}
	F_c \arrow{r}{F(f)} \arrow{d}{\psi_{c}}
	&F_{c'} \arrow{d}{\psi_{c'}}\\
	G_{c} \arrow{r}{G(f)} &G_{c'}.
\end{tikzcd}
\end{center}
Natural transformations $\psi:F\rightarrow G$ are considered as morphisms in the category $\mathcal{D}^\mathcal{C}$ of all functors from $\mathcal{C}$ to $\mathcal{D}.$

\paragraph{The interleaving distance between $\R^d$ (or $\Z^d$)-indexed functors.} In what follows, for any $\eps\in [0,\infty)$, we will denote the vector $\eps(1,\ldots,1)\in \R^d$ by $\vec{\eps}$. The dimension $d$ will be clearly specified in context.

\begin{definition}[$\mathbf{v}$-shift functor]\label{def:eps-shifting} Let $\mathcal{C}$ be any category.  For each $\mathbf{v}\in [0,\infty)^n$,  the $\bv$-shift functor $(-)(\bv):\C^{\R^d}\rightarrow \C^{\R^d}$ is defined as follows: 
\begin{enumerate}[label=(\roman*)]
    \item (On objects) Let $F:\R^d\rightarrow \mathcal{C}$ be any functor. Then the functor $F(\mathbf{v}):\R^d\rightarrow \mathcal{C}$ is defined as follows: For any $\ba\in \R^d$, 
	\[F(\mathbf{v})_\ba:=F_{\ba+\bv}.\]
	Also, for another $\ba'\in \R^d$ such that $\ba\leq \ba'$ we define
	\[\varphi_{F(\mathbf{v})}(\ba,\ba'):= \varphi_{F}\left(\ba+\bv, \ba'+\bv\right).\]
In particular, if $\mathbf{v}=\vec{\eps}\in[0,\infty)^d$, then we simply write $F(\eps)$ in lieu of $F(\vec{\eps})$.	

    \item (On morphisms) Given any natural transformation $\psi:F\Rightarrow G$, the natural transformation $\psi(\bv):F(\bv)\Rightarrow G(\bv)$ is defined as $\psi(\bv)_\ba=\psi_{\ba+\bv}:F(\bv)_\ba\rightarrow G(\bv)_\ba$ for each $\ba\in \R^d$.
\end{enumerate}
 
\end{definition}

For any $\bv\in [0,\infty)^d$, let $\psi_F^\bv: F \Rightarrow F(\bv)$ be the natural transformation whose restriction to each $F_\ba$ is the morphism $\varphi_F(\ba, \ba+\bv)$ in $\C$. When $\bv=\vec{\eps}$, we denote $\psi_F^\bv$ simply by $\psi_F^\eps$.  
\begin{definition}[$\bv$-interleaving between $\R^d$-indexed functors]\label{def:interleaving_general} Let $\mathcal{C}$ be any category. Given any two functors $F,G:\R^d\rightarrow \mathcal{C},$ we say that they are \emph{$\bv$-interleaved} if there are natural transformations $f:F\Rightarrow G(\bv)$ and $g:G\Rightarrow F(\bv)$  such that 
	\begin{enumerate}[label=(\roman*)]
		\item $g(\bv)\circ f = \psi_F^{2\bv}$,
		\item $f(\bv)\circ g = \psi_G^{2\bv}$.
	\end{enumerate} 

In this case, we call $(f,g)$ a \emph{$\bv$-interleaving pair}. When $\bv=\eps(1,\ldots,1)$, we simply call $(f,g)$ \emph{$\eps$-interleaving pair}. The interleaving distance between $\dint^\C$ is defined as
\begin{equation}\label{eq:interleaving0}
    \dint^\C(F,G):=\inf\{\eps\in [0,\infty):F,G\ \mbox{are $\vec{\eps}$-interleaved}\},
\end{equation}
	where we set $\dint^\C(F,G)=\infty$ if there is no $\eps$-interleaving pair between $F$ and $G$ for any $\eps\in [0,\infty)$. Then $\dint^\C$ is an extended pseudo-metric for $\C$-valued  $\R^d$-indexed functors.  \textbf{By replacing $\R^d$ by $\Z^d$ in Definitions \ref{def:eps-shifting} and \ref{def:interleaving_general}, we similarly obtain the interleaving distance between $\Z^d$-indexed functors.}
\end{definition}

\begin{definition}[Poset $\U$]\label{item:posets1}
  The poset $\U:=\left\{(u_1,u_2)\in \R^2: u_1\leq u_2\right\}$ is equipped with the partial order inherited from $\R^{\mathrm{op}}\times \R$ (see Figure \ref{fig:kan poset} (A)), i.e. $(u_1,u_2)\leq (u_1',u_2')$ in $\U$ if and only if the interval $(u_1,u_2)\subset \R$ is contained in $(u_1',u_2')\subset \R$.

\end{definition}

The reflection map $\mathbf{r}:\R^{\mathrm{op}}\rightarrow \R$ defined by $t\mapsto -t$ induces a poset isomorphism  $\R^\mathrm{op}\times \R\rightarrow \R^2$ and this in turn induces an isomorphism $\C^{\R^2}\rightarrow \C^{\R^{\mathrm{op}}\times \R}.$ Therefore, the notions of $\eps$-interleaving pair and the interleaving distance $\dint^\C$ on $\ob(\C^{\R^2})$ carry over to $\R^\mathrm{op}\times \R$-indexed or $\U$-indexed functors. 
\paragraph{Extension functor and the interleaving for $\ZZ$-indexed functors.} From Definition \ref{def:posets} \ref{item:ZZ}, recall the poset $\ZZ=\{(i,j)\in \Z^2: j=i\ \mbox{or}\ j=i-1\}\subset\R^{\mathrm{op}}\times \R$. Let $\iota:\ZZ\hookrightarrow \R^\mathrm{op}\times \R$ be the canonical inclusion map. For any $\bu=(u_1,u_2)\in \U$, let
\[\ZZ[\iota\leq \bu]:=\{a\in \ZZ: \iota(a)\leq \bu\} 
\] 
which is a subposet of $\ZZ$. Observe that $\ZZ[\iota\leq \bu]$ cannot be empty for any choice of $\bu\in \U$. See Figure \ref{fig:kan poset} (B).

\begin{figure}
    \centering
    \includegraphics[width=0.8\textwidth]{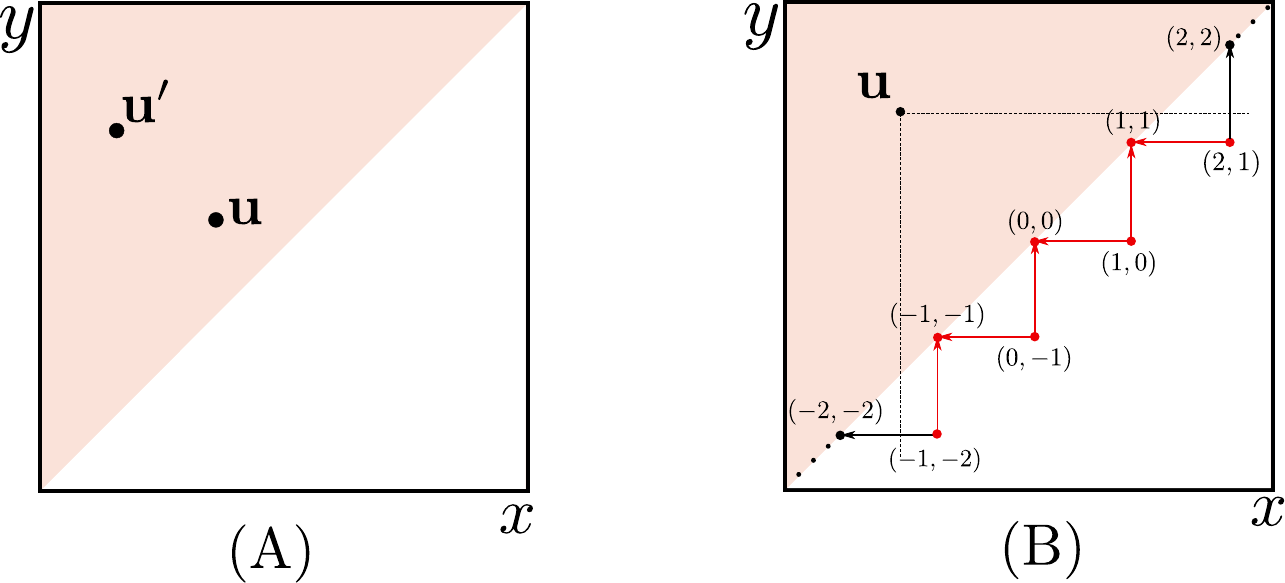}
    \caption{ (A) The shaded region stands for the poset $\U$. For $\bu,\bu'\in \U$ as marked in the figure, we have $\bu\leq\bu'$. (B) Fixing $\bu\in\U$ as shown, the subposet $\ZZ[\iota\leq \bu]$ is indicated by  the red points and the red arrows.}
    \label{fig:kan poset}
\end{figure}

 We review the definition of extension functor $E:\C^{\ZZ}\rightarrow \C^{\U}$ of \cite{botnan2018algebraic} for a cocomplete category $\C$.  For any $M: \ZZ\rightarrow \C$, define the functor $\tilde{E}(M):\R^\mathrm{op}\times \R\rightarrow \C$ as follows: For $\ba\in \R^\mathrm{op}\times \R$,  \[\tilde{E}(M)(\ba):=\varinjlim M|_{\ZZ[(\iota \leq \ba)]}.\]
  Also, for any pair $\ba\leq \bb$ in $\R^{\mathrm{op}}\times \R$, since $\ZZ[\iota\leq \ba]$ is a subposet of $\ZZ[\iota\leq\bb]$, the linear map $\varinjlim M|_{\ZZ[(\iota \leq \ba)]}\rightarrow \varinjlim M|_{\ZZ[(\iota \leq \bb)]}$, is uniquely specified by the initial property of the colimit $\varinjlim M|_{\ZZ[(\iota \leq \ba)]}$ (Definition \ref{def:colimit} and Remark \ref{rem:restriction}).  This $\tilde{E}(M)$ is called the \emph{left Kan extension} of $M$ along $\iota:\ZZ\hookrightarrow \R^{\mathrm{op}}\times \R$. Given $M,N:\ZZ\rightarrow \C$ and a natural transformation $\Gamma:M\rightarrow N$, universality of colimits also yields an induced morphism $\tilde{\Gamma}:\tilde{E}(M)\rightarrow\tilde{E}(N)$. Given any $M:\ZZ\rightarrow \C$,  the functor $E(M):\U\rightarrow \C$ is defined as the restriction $\tilde{E}(M)|_\mathbf{U}: \U\rightarrow \C$.

\begin{definition}[Interleaving distance between zigzag modules \cite{botnan2018algebraic}]\label{def:interleaving by botnan}  Let $\C$ be a cocomplete category. For any $M,N:\ZZ\rightarrow \C$, 
\[\dint(M,N):=\dint^{\C}\left(E(M),E(N)\right).\]
\end{definition}

\subsection{Existing stability theorems}\label{sec:existing stability}
In this section we review the existing stability results from \cite{botnan2018algebraic,mccleary2018bottleneck}.

\paragraph{Stability theorem of Patel.} We recall \cite[Theorem 5.6]{mccleary2018bottleneck}: 

\begin{theorem}[Stability for constructible persistence modules ]\label{thm:patel}Let $\C$ be an essentially small, abelian\footnote{A category $\C$ is \emph{abelian} if morphisms and objects in $\C$ can be ``added" and \emph{kernels} and \emph{cokernels} exist in $\C$ with some desirable properties. See \cite[p.198]{mac2013categories} for the precise definition.} category. Let $F,G:\R\rightarrow \C$ be constructible (Definition \ref{def:constructible}). Then,
\[\bott\left(\dgm_{\C}(F),\dgm_{\C}(G)\right)\leq \dint(F,G).
\]
We remark that the persistence diagram $\dgm_{\C}(F)$ can be computed via the following process: (1) Re-indexing of $F:\R\rightarrow \C$ to obtain $D(F):\Z\rightarrow \C$ (Section \ref{sec:patel}), (2) computing the persistence diagram of $D(F)$ by Definition \ref{def:generalized PD2}, and (3) the rescaling of the persistence diagram (Section \ref{sec:patel}).
\end{theorem}

\paragraph{Bottleneck distance.} 
 We regard the extended real line $\ER:=\R\cup\{-\infty,\infty\}$ as a poset with the canonical order $\leq$.  Let $\EU:=\left\{(u_1,u_2)\in \ER^2: u_1\leq u_2\right\}$ be equipped with the partial order inherited from $\ER^{\mathrm{op}}\times \ER$.
 For $\bu=(u_1,u_2),\ \bv=(v_1,v_2)\in \EU$, let\[\norm{\bu-\bv}_\infty:=\max\left(\abs{u_1-v_1}, \abs{u_2-v_2}\right).
\]  
We can quantify the difference between two persistence diagrams or two barcodes of real intervals, using the bottleneck distance \cite{cohen2007stability}: 

\begin{definition}[The bottleneck distance ]\label{def:bottleneck distance}Let $X_1,X_2$ be multisets of points in $\EU$. Let $\alpha:X_1\nrightarrow X_2$ be a matching, i.e. a partial injection. We call $\alpha$ an \emph{$\eps$-matching} if 
\begin{enumerate}[label=(\roman*)]
    \item for all $\bu\in \dom(\alpha)$, $\norm{\bu-\alpha(\bu)}_\infty\leq \eps$,
    \item for all $\bu=(u_1,u_2)\in X_1\setminus \dom(\alpha)$, $u_2-u_1 \leq 2\eps$,
    \item for all $\bv=(v_1,v_2)\in X_2 \setminus \im(\alpha)$, $v_2-v_1 \leq 2\eps$.
\end{enumerate}
Their bottleneck distance $\bott(X_1,X_2)$ is defined  as the infimum of $\eps\in[0,\infty)$ for which there exists an $\eps$-matching $\alpha:X_1\nrightarrow X_2$.
\end{definition}

Recall that $\langle b,d \rangle_{\ZZ}$ for $b,d\in\Z$ denotes intervals of $\ZZ$ and that  $\langle b,d \rangle$ for $b,d\in\R$ denotes intervals of $\R$.
\begin{Remark}(1) Given a pair of multisets of intervals $\langle b,d \rangle_{\ZZ}$ of $\ZZ$, their bottleneck distance is defined by converting those multisets into the multisets of $\EU$ via the identification $\langle b,d \rangle_{\ZZ}\leftrightarrow (b,d)\in \EU$. (2) A map $\inter(\ZZ)\rightarrow \Z_+$ can be considered as a multiset of $\inter(\ZZ)$ in an obvious way, and thus such maps can also be compared in $\bott$. (3) The bottleneck distance between multisets of real intervals $\langle b,d \rangle$ is also defined  via the identification $\langle b,d \rangle \leftrightarrow (b,d)\in \EU$.  
\end{Remark}

\paragraph{Algebraic stability for zigzag modules.}

\begin{theorem}[Bottleneck stability for zigzag modules {\cite{bjerkevik2016stability,botnan2018algebraic}}]\label{thm:bottleneck stability}For any $M,N:\ZZ\rightarrow \vect$,
\[\bott\left(\dgmzz(M),\dgmzz(N)\right)\leq 2\cdot \dint(M,N).\]
\end{theorem}

Let $M$ be any zigzag module. By $\dgmzz_{\mathbf{o}}(M),\dgmzz_{\mathbf{co}}(M),\dgmzz_{\mathbf{oc}}(M)$ and $\dgmzz_{\mathbf{c}}(M)$, we mean the subcollection of $\dgmzz(M)$, consisting solely of the points of the form  $(b,d)_{\ZZ}, [b,d)_{\ZZ}, (b,d]_{\ZZ},$ and $[b,d]_{\ZZ}$, respectively.

\begin{Remark}\label{rem:forget-deco} Suppose that the two zigzag modules $M,N$ in Theorem \ref{thm:bottleneck stability} are the \emph{levelset zigzag persistent homology} of any two \emph{Morse type} functions $f,g:X\rightarrow \R$ (Definition \ref{def:morse type}). Then, the inequality in Theorem \ref{thm:bottleneck stability} can be extended and strengthened as follows \cite[Theorem 4.11]{botnan2018algebraic,carlsson2019parametrized}: For each $\star\in \{\mathbf{o},\mathbf{co},\mathbf{oc},\mathbf{c}\}$,
\[\bott\left(\dgmzz_\star(M),\dgmzz_\star(N)\right)\leq \dint(M,N) \leq \norm{f-g}_\infty.\]
Therefore, we also have \[\bott\left(\dgmzz(M),\dgmzz(N)\right)\leq \max_{\star}\bott\left(\dgmzz_\star(M),\dgmzz_\star(N)\right)\leq \dint(M,N)\leq \norm{f-g}_\infty,\]where the maximum is taken over all $\star\in \{\mathbf{o},\mathbf{co},\mathbf{oc},\mathbf{c}\}$.  
\end{Remark}

\section{Proof of Proposition \ref{rem:implication}}\label{sec:proof}

In this section we prove Proposition \ref{rem:implication}. For notational simplicity, we set $\Pb=\ZZ$. We remark that the proof extends to arbitrary connected locally finite posets.  

\paragraph{Canonical bases and $l$-type intervals.} Suppose that $M:\ZZ\rightarrow \vect$ is given as  $M=\bigoplus_{c\in C} I^{J_c}$ for an index set $C$.  Then  for each $(i,j)\in \ZZ$, the dimension of $M_{(i,j)}$ is the total multiplicity of intervals containing $(i,j)$ in $\dgmzz(M)=\lmulti J_c: c\in C\rmulti$. Given any $c\in C$, and $(i,j)\in J_c$, let $e^{(i,j)}_c$ denote the $1$ in the field $\F$ which corresponds to the $(i,j)$-component of $I^{J_c}$. Then for each $(i,j)\in \ZZ$, the vector space $M_{(i,j)}$ admits the \emph{canonical basis} $B_{(i,j)}=\left\{e^{(i,j)}_c: J_c \ \mbox{contains $(i,j)$} \right\}$ and hence every $v\in M_{(i,j)}$ can be uniquely expressed as a linear combination of the elements in $B_{(i,j)}$, i.e.
\begin{equation}\label{eq:canonical expression}
    v=\sum_{J_c\ni (i,j)} a_c e_c^{(i,j)},\hspace{5mm}  a_c\in \F.
\end{equation}
This expression is called the \emph{canonical expression} of $v$. 
The collection $\B=(B_{(i,j)})_{(i,j)\in \ZZ}$ is called the \emph{canonical basis} of $M=\bigoplus_{c\in C} I^{J_c}$.

In order to prove Proposition \ref{rem:implication}, we identify a certain type of intervals of the poset $\ZZ$:

\begin{definition}[$l$-type intervals]\label{def:l-type} Any interval $I$ of $\ZZ$ is called \emph{$l$-type}\footnote{Given any $M:\ZZ\rightarrow \vect$, $l$-type intervals in $\dgmzz(M)$ contribute to the dimension of the \emph{limit} of $M$ and that is the reason for the name `$l$'-type.} if $I$ is of the form $(b,d)_\ZZ$ for some $b\in \Z\cup\{-\infty\}$ and $d\in \Z\cup \{\infty\}$ (see Figure \ref{fig:intervals}).

\end{definition}

\begin{notation}\label{def:equivalence} Let $M:\ZZ\rightarrow \vect$ be any zigzag module and let $(i,j),(i',j')\in \ZZ$. For $v_{(i,j)}\in M_{(i,j)}$ and $v_{(i',j')}\in M_{(i',j')}$, we write $v_{(i,j)}\sim v_{(i',j')}$ if $(i,j),(i',j')$ are comparable and either 
\begin{align*}
   & v_{(i,j)}=\varphi_M((i',j'), (i,j))(v_{(i',j')})\ \mbox{(when $(i',j')\leq (i,j)$) \ \ or \ \ }\\
    & v_{(i',j')}=\varphi_M((i,j), (i',j'))(v_{(i,j)})\ \mbox{(when $(i,j)\leq (i',j')$).}
\end{align*}

\end{notation}

\begin{proof}[Proof of Proposition \ref{rem:implication}] For simplicity of notation, we prove the proposition when $\Pb=\ZZ$ and $J=(-\infty,\infty)=\ZZ$. The proof straightforwardly extends to any connected, locally finite poset $\Pb$ and $J\in \con(\Pb)$. 
By Theorem \ref{thm:interval decomposition0}, we may assume that $M=\bigoplus_{c\in C}I^{J_c}$, where $\dgmzz(M)=\lmulti J_c:c\in C\rmulti$. In what follows, we identify $\ZZ$ with the integers $\Z$ via the bijection $(i,j)\mapsto i+j$. Therefore, by $\bigoplus_{k\in\Z} M_{k}$ and $\prod_{k\in\Z} M_{k}$, we will denote $\bigoplus_{(i,j)\in\ZZ} M_{(i,j)}$ and $\prod_{(i,j)\in\ZZ} M_{(i,j)}$ respectively.  

	\begin{enumerate}[label=(\roman*)]
		\item \label{item:limit}The limit of $M$ is (isomorphic to) the pair $\left(V,(\pi_k)_{k\in \Z}\right)$ described as follows: 
		\begin{equation}\label{eq:limit}
		    V:=\left\{(v_k)_{k\in \Z}\in \prod_{k\in \Z} M_k:\ \forall k\in \Z,\ v_{k}\sim v_{k+1}   \right\}.
		\end{equation}
		For each $k\in \Z$, the map $\pi_k:V\rightarrow M_k$ is the canonical projection.

		\item The colimit of $M$ is (isomorphic to) the pair $\left(U, (i_k)_{k\in \Z}\right)$ described as follows: $U$ is the quotient vector space $\left(\bigoplus_{k\in\Z} M_k\right)/W$, where $W$ is the subspace of the direct sum $\bigoplus_{k\in \Z} M_k$ which is generated by the vectors of the form  $(\cdots,0,\ldots,0, v_k, -v_{k+1},0,\ldots,0,\cdots)$ with $v_k\sim v_{k+1}$ (Notation \ref{def:equivalence}). Let $q$ be the quotient map from  $\bigoplus_{k\in \Z} M_k$ to $U=\left(\bigoplus_{k\in \Z}M_k\right)/W$. For $k\in \Z$, let the map $\bar{i_k}:M_k\rightarrow \bigoplus_{k\in \Z} M_k$ be the canonical injection. Then $i_k:M_k\rightarrow U$ is the composition $q\circ \bar{i_k}$. 
		\label{item:colim}
	\end{enumerate}

Let $0\in \Z$. Since the canonical map $\psi_M:\varprojlim M \rightarrow \varinjlim M$ is equal to $i_0 \circ \pi_0$, it suffices to show that the dimension of the vector space $\im (i_0\circ \pi_0)$ is equal to the cardinality of $(-\infty,\infty)_\ZZ$ in $\dgmzz(M)$. 

If $M_0$ is the zero space, then it is clear that there is no $(-\infty,\infty)_\ZZ$ in $\dgmzz(M)$ and that $\im(\pi_0)=0$. Therefore, $\im(i_0\circ \pi_0)=0$, and the statement directly follows.

Assume that $M_0$ is not trivial.  Define 
\[C_0:=\{c\in C:\ 0\in J_c\}.\]
Since $M_0$ is a finite dimensional vector space, $\abs{C_0}=\dim(M_0)$ is finite and hence we can write $C_0=\{c_1,\ldots,c_m\}$, where $\dim(M_0)=m$. Also, we can write $M_0\cong \bigoplus_{j=1}^m\F_j$, where each $\F_j=\F$ is the component of the interval module $I^{J_{c_j}}$ at $0\in\ZZ$. We identify $M_0$ with $ \bigoplus_{j=1}^m\F_j$. Then, we claim that
\begin{equation}\label{eq:claim2}
    \im (\pi_0)=\left\{(a_1,\ldots,a_m)\in \bigoplus_{j=1}^m\F_j:\ a_j=0\ \mbox{if $J_{c_j}$ is not $l$-type}\right\}.
\end{equation}
We prove this equality at the end of the proof. 

For $j=1,\ldots,m$, let $e_j:=(0,\ldots,0,\underset{j-\mbox{th}}{1},0,\ldots,0)\in \bigoplus_{j=1}^m\F_j$. By equation (\ref{eq:claim2}), the set $B_0=\{e_j: \mbox{$J_{c_j}$ is $l$-type}\}$ is a basis of $\im(\pi_0)$. Therefore, the dimension of $\im(i_0\circ \pi_0)$ is equal to the dimension of the space that is spanned by the image of $B_0$ under the map $i_0:M_0\rightarrow U$. By invoking item \ref{item:colim} above, if $J_{c_j}\neq (-\infty,\infty)_\ZZ$, it follows that $i_0(e_j)=0\in U$. 

Let $C_0^{\mbox{full}}:=\{c\in C: J_c=(-\infty,\infty)_{\ZZ}\}$, which is a subset of $C_0=\{c_1,\ldots,c_m\}$. Assuming that $C_0^{\mbox{full}}\neq \emptyset$, suppose that $C_0^{\mbox{full}}=\{c_1,\ldots,c_n\}$ for some $n\leq m$ without loss of generality. Invoking item \ref{item:colim} above, the set $\{i_0(e_{1}),\ldots, i_0(e_{n})\}$ is linearly independent in $U$. 
Therefore, we have that
\[\rank(i_0\circ \pi_0)=n=(\mbox{the multiplicity of $(-\infty,\infty)_\ZZ$ in $\dgmzz(M)$)}, \] 
as desired.

Finally we prove equation (\ref{eq:claim2}). First we prove ``$\subset$". Recall item \ref{item:limit} above and pick any $v=(v_k)_{k\in \Z}\in V$. Then $\pi_0(v)=v_0=(a_1,\ldots,a_m)$. Suppose that $c_j\in C_0$ is such that $J_{c_j}$ is not $l$-type. This implies that the interval $J_{c_j}$ has an endpoint $r=(r_1,r_2)\in \ZZ$ where $r_1=r_2\in \Z$ and then either $(r_1+1,r_1)$ or $(r_1,r_1-1)$ is not in $J_{c_1}$. Without loss of generality, assume that $s=(r_1+1,r_1)\in \ZZ$ does not belong to $J_{c_1}$. By the choice of $v=(v_k)_{k\in \Z}$, we have $v_0\sim v_1\sim \ldots \sim v_{2r_1}\sim v_{2r_1+1}$, and this leads to that $a_j=0$.

Next we show ``$\supset$". Pick any $(a_1,\ldots,a_m)$ in the RHS of equation (\ref{eq:claim2}). For each $k\in \Z$, define $v_k\in M_k$ using the canonical expression:
\[v_k:=\sum_{J_c\ni k} b_c e^{k}_{c},\hspace{5mm}  b_c\in \F,\]
where $b_c=0$ if $c\not\in C_0$ and $b_c=a_j$ if $c=c_j\in C_0=\{c_1,\ldots,c_m\}$. Let $v:=(v_k)_{k\in \Z}$. Then, one can check that $\pi_0(v)=(a_1,\ldots,a_m)$, completing the proof.
\end{proof}

\section{Constructible persistence modules and re-indexing}\label{sec:patel}

Patel generalizes the persistence diagram of Cohen-Steiner, Edelsbrunner, and Harer to the setting of \emph{constructible} $\R$-indexed diagrams valued in a symmetric monoidal category \cite{patel2018generalized}.

\begin{definition}[Constructible $\R$-indexed diagrams]\label{def:constructible} Let $S=\{s_1<s_2<\ldots<s_n\}$ be a finite set of $\R$. A diagram $F:\R\rightarrow \C$ is $S$-\emph{constructible} if 
\begin{enumerate}[label=(\roman*)]
    \item for $p\leq q<s_1$, $\varphi_F(p, q)$ is the identity on $e$,
    \item for $s_i\leq p\leq q< s_{i+1}$, $\varphi_F(p, q)$ is an isomorphism,
    \item for $s_n\leq p\leq q$, $\varphi_F(p, q)$ is an isomorphism.
\end{enumerate}
If $G:\R\rightarrow \C$ is $T$-constructible for some finite set $T\subset \R$, then we call $G$ \emph{constructible}.
\end{definition}

\paragraph{Re-indexing an $\R$-indexed diagram by $\Z$.} Let $F:\R\rightarrow \C$ be $S$-constructible with $S=\{s_1<s_2<\ldots<s_n\}$. A functor $D(F):\Z\rightarrow \C$ that contains all the algebraic information of $F$ would be defined as follows:

\begin{align*}
    D(F)_{i}&=\begin{cases}
    e,&\mbox{for $i\leq 0$,}\\
    F_{s_i},&\mbox{for $i=1,\ldots,n$,}\\
    F_{s_n},&\mbox{for $i\geq n+1$.}
    \end{cases}\\
    \varphi_{D(F)}\left(i,i+1\right)&=\begin{cases} \id_e,&\mbox{for $i\leq 0$,}\\
     \varphi_F(s_i,s_{i+1}),&\mbox{for $i=0,1,\ldots,n-1$, where $s_0$ is arbitrarily chosen in $(-\infty,s_1)$,}\\
     \id_{F_{s_n}},&\mbox{for $i\geq n$.}
    \end{cases}
\end{align*}
    
When $\C=\vect$, there exists a bijection from the barcode of $F$ to that of $D(F)$ via $[s_i,s_j)\mapsto [i,j-1]$ for $1\leq i<j\leq n$, and $[s_i,\infty) \mapsto [i,\infty)$ for $1\leq i\leq n$.
    
\paragraph{Re-indexing a $\Z$-indexed diagram by $\ZZ$.} Let $F:\Z\rightarrow \C$. Let us define $L(F):\ZZ\rightarrow \C$ as follows \cite[Remark 4.5]{botnan2018algebraic}:

\begin{align*}
    L(F)_{(i,i)}=L(F)_{(i+1,i)}&=F_i\\
    \varphi_{L(F)}\left((i+1,i),(i,i)\right)&=\id_{F_{i}}\\
    \varphi_{L(F)}\left((i-1,i),(i,i)\right)&=\varphi_F(i-1, i).
\end{align*}
When $\C=\vect$ and $F_i=0$ for $i\leq 0$, there exists a bijection from the barcode of $F$ to that of $L(F)$ via $[a,b]\mapsto [a,b+1)_{\ZZ}$  for $a,b\in \Z$ with $a\leq b$, and $[a,\infty)\mapsto [a,\infty)_{\ZZ}$ for $a\in \Z$.

\section{Rigorous definition of Reeb graphs}\label{appendix:Reeb graphs}

In order to introduce the definition of Reeb graphs, we begin by introducing the notion of \emph{Morse type} functions from \cite{botnan2018algebraic,carlsson2009zigzag}. 

\begin{definition}[Morse type functions]\label{def:morse type} Let $X$ be a topological space. We say that a continuous function \mbox{$p:X\rightarrow \R$} is of \emph{Morse type} if
	\begin{enumerate}[label=(\roman*)]
		\item There exists a strictly increasing function $\G:\Z\rightarrow \R$ such that $\lim_{i\rightarrow +\infty}\G(i)=\infty$, $\lim_{i\rightarrow -\infty}\G(i)=-\infty$ and such that for each open interval $I_i=(\G(i),\G(i+1))$ there exist a topological space $Y_i$  and a homeomorphism $h_i:I_i\times Y_i\rightarrow p^{-1}(I_i)$ with $f\circ h_i$ being the projection $I_i\times Y_i \rightarrow I_i$. \label{item:morse type1}
		\item Each homeomorphism $h_i:I_i\times Y_i \rightarrow p^{-1}(I_i)$ extends to a continuous function
		$$\bar{h}_i:\bar{I_i}\times Y_i \rightarrow p^{-1}(\bar{I_i}),$$where $\bar{I_i}$ denotes the closure of $I_i$.
		\item For all $t\in \R$ and $k\in\Z_+$, $\dim\Hrm_k\left(p^{-1}(t)\right)<\infty$. 
	\end{enumerate} 
\end{definition}

We introduce the definition of Reeb graphs \cite{de2016categorified}.

\begin{definition}[Reeb graphs]\label{def:Reeb graph}
		Let $X$ be a topological space and let \hbox{$p:X\rightarrow \R$} be of Morse type. If the topological spaces $Y_i$ as in Definition \ref{def:morse type} \ref{item:morse type1} are finite sets of points with the discrete topology, then the pair $(X,p)$ is said to be a \emph{Reeb graph}.\footnote{In \cite{de2016categorified}, the authors use the term \emph{$\R$-graph} instead of Reeb graph.} See Figure \ref{fig:Reeb} (A) for an illustrative example. \label{item:r-graph1}
\end{definition}

\input{stability.tex}

\paragraph{Acknowledgements}
The idea of studying the map from the limit to the colimit of a given diagram stems from work by Amit Patel and Robert MacPherson circa 2012. We are grateful to Amit Patel and Justin Curry for insightful comments. We thank Michael Lesnick, Magnus Bakke Botnan, and anonymous reviewers for suggesting alternative proofs of (variants of) Proposition \ref{rem:implication}. We also thank Amit Patel and Alex McCleary for suggesting Examples \ref{ex:Amit} and  \ref{ex:Alex}, respectively.  Many thanks to Justin Curry and Benedikt Fluhr for finding an error in an earlier draft and for suggesting us to consider the Reeb graph in Figure \ref{fig:Curry}. Lastly, WK thanks Peter Bubenik, Nicolas Berkouk and Osman Okutan for beneficial discussions.  
This work was partially supported by NSF grants IIS-1422400, CCF-1526513, DMS-1723003, and CCF-1740761.

\bibliography{biblio}
\bibliographystyle{abbrv}
\end{document}

%% file: stability.tex
\section{Stability of the generalized rank invariant}
In this section we prove Theorem \ref{thm:stability of gen rank} which says that the generalized rank invariant $\rk(M)$ is stable in the erosion distance (Definition \ref{def:erosion}) under perturbations of $M:\Z^n\rightarrow \vect$ in the sense of the interleaving distance (Definition \ref{def:interleaving_general}).

\begin{Remark}\label{remark:computing limits and colimits}
Given any $M:\Pb\rightarrow \Vect$, we can compute the limit and colimit of $M$ as follows:

	\begin{enumerate}[label=(\roman*)]
		\item \label{item:limit}The limit of $M$ is (isomorphic to) the pair $\left(V,(\pi_k)_{k\in \Pb}\right)$ described as follows: 
		\begin{equation}\label{eq:limit}
		    V:=\left\{(v_k)_{k\in \Pb}\in \prod_{k\in \Pb} M_k:\ \forall k\leq l \ \mbox{in}\ \Pb, \varphi_M(k,l)(v_k)=v_{l}   \right\}.
		\end{equation}
		For each $k\in \Pb$, the map $\pi_k:V\rightarrow M_k$ is the canonical projection. An element of $V$ is called a \textbf{section} of $M$.

		\item The colimit of $M$ is (isomorphic to) the pair $\left(U, (i_k)_{k\in \Pb}\right)$ described as follows: For $k\in \Pb$, let $\bar{i_k}:M_k\rightarrow \bigoplus_{k\in \Pb} M_k$ be the canonical injection. $U$ is the quotient space $\left(\bigoplus_{k\in\Pb} M_k\right)/W$, where $W$ is the subspace of $\bigoplus_{k\in \Pb} M_k$ which is generated by the elements of the form $\bar{i}_k(v_k)-\bar{i}_{l}(v_l)$ for $k\leq l$ in $\Pb$ and $v_k\in M_k$ and $v_l\in M_l$. Let $q$ be the quotient map from  $\bigoplus_{k\in \Pb} M_k$ to $U=\left(\bigoplus_{k\in \Pb}M_k\right)/W$. For $k\in \Pb$, $i_k:M_k\rightarrow U$ is the composition $q\circ \bar{i_k}$. 
		\label{item:colim}
	\end{enumerate}
\end{Remark} 

\begin{definition}Let $I\subset \Z^n$ and $\eps\in \Z_+$. Let us define the \textbf{$\eps$-thickening} of $I$  as
\[I^\eps:=\left\{p\in \Z^n: \mbox{there exists $q\in I$ such that} \norm{p-q}_\infty\leq \eps \right\}.\]
\end{definition}

Let $\C$ be any essentially small and symmetric monoidal category satisfying Convention \ref{convention} \ref{item:a full subcategory of a bicomplete category}. We generalize the erosion distance in \cite{patel2018generalized} as follow.

\begin{definition}\label{def:erosion} Given $M,N:\Z^n\rightarrow \C$, the \textbf{erosion distance} between $\rk(M)$ and $\rk(N)$ is \[\dero(\rk(M),\rk(N)):=\min\{\eps\in[0,\infty]:\forall I\in \inter(\Z^n), \rk(M)(I)\geq \rk(N)(I^\eps) \ \mbox{and} \ \rk(N)(I)\geq \rk(M)(I^\eps)\}.\] 
\end{definition}

Note that $\dero(\rk(M),\rk(N))=0$ if and only if $\rk(M)=\rk(N)$ on $\inter(\Z^n)$. The triangle inequality can be directly proved by utilizing the following proposition and the monotonicity of the rank invariant (Proposition \ref{prop:order reversing property}).

\begin{proposition}\label{prop:thickening} Let $I\in \inter(\Z^n)$. For $\eps,\delta\in \Z_+$, $(I^\eps)^\delta=I^{\eps+\delta}$.
\end{proposition}

\begin{proof}
$(\subset)$ Let $p\in (I^\eps)^\delta$. Then, there exists $q\in I^\eps$ such that $\norm{p-q}_{\infty}\leq \delta$. Also, there exists $r\in I$ such that $\norm{q-r}_{\infty}\leq \eps$. Then, $\norm{p-r}_{\infty}\leq \norm{p-q}_{\infty}+\norm{q-r}_{\infty}\leq \delta+\eps$, implying that $p\in I^{\eps+\delta}$.

$(\supset)$ Let $p\in I^{\delta+\eps}$. Then, there exists $r\in I$ such that $\norm{p-r}_{\infty}\leq \delta+\eps$. For $i=1,\ldots,n$, let $p_i$ and $r_i$ be the $i$-th coordinates of $p$ and $r$, respectively. For each $i$, since $\abs{p_i-r_i}\leq \eps+\delta$, there exists $q_i\in \Z$ such that $\abs{p_i-q_i}\leq \delta$ and $\abs{q_i-r_i}\leq \eps$. Let $q:=(q_1,\ldots,q_n)\in \Z^n$. Then, $\norm{p-q}_{\infty}\leq \delta$ and $\norm{q-r}_{\infty}\leq \eps$, implying that $p\in (I^{\eps})^\delta$.
\end{proof}

In what follows, we prove that the generalized rank invariant $\rk(M)$ is stable in $\dero$ under perturbations of $M:\Z^n\rightarrow \vect$.

\begin{theorem}\label{thm:stability of gen rank} For any $M,N:\Z^n\rightarrow \vect$, we have $\dero(\rk(M),\rk(N))\leq \dint(M,N)$.
\end{theorem}

\begin{Remark}
\begin{enumerate}[label=(\roman*)]
    \item Theorem \ref{thm:stability of gen rank} can  be straightforwardly extended to the setting of $\R^n$-indexed persistence modules. Also, the target category $\vect$ can be replaced by  other categories including $\sets$ and $\ab$. 
    \item The computational complexity of $\dero(\rk(M),\rk(N))$ is analyzed (in \cite{kim2021computing} in a suitable setting).
  
    \item  In Definition \ref{def:erosion}, Proposition \ref{prop:thickening}, and Theorem \ref{thm:stability of gen rank}, there is no obstacle in replacing $\inter(\Z^n)$ by either the larger collection $\con(\Z^n)$ or the smaller collection $\{[p,q]\subset \Z^n:p\leq q \ \mbox{in $\Z^n$}\}$.
\end{enumerate}
\end{Remark}

\begin{proof}[Proof of Theorem \ref{thm:stability of gen rank}] If $\dint(M,N)=\infty$, there is nothing to prove. Assume that $\dint(M,N)=\eps<\infty$, and let $\alpha:M\Rightarrow N(\vec{\eps})$ and $\beta:N\Rightarrow M(\vec{\eps})$ be an $\eps$-interleaving pair (Definition \ref{def:interleaving by botnan}).  Fix $I\in \inter(\Z^n)$. We only prove that $\rk(N)(I)\geq \rk(M)(I^\eps)$. To this end, it suffices to find linear maps $\tilde{\alpha}$ and $\tilde{\beta}$ which make the diagram below commute

\begin{equation}\label{eq:commutative diagram}
    \begin{tikzcd}
\varprojlim M|_{I^\eps} \arrow[r] \arrow[d, "\tilde{\alpha}"]
& \varinjlim M|_{I^\eps} \\
\varprojlim N|_{I} \arrow[r]
& \varinjlim N|_{I} \arrow[u, "\tilde{\beta}"] 
\end{tikzcd}
\end{equation}
where the two horizontal maps are the canonical limit-to-colimit maps. Let us concretize each of the vector spaces $\varprojlim M|_{I^\eps}$, $\varinjlim M|_{I^\eps}$, $\varprojlim N|_{I}$ and $ \varinjlim N|_{I}$ as described in Remark \ref{remark:computing limits and colimits}. Then, let us define $\tilde{\alpha}$ as $(v_p)_{p\in I^\eps} \mapsto (\alpha_p(v_p))_{p+\vec{\eps}\in I}=(\alpha_{q-\vec{\eps}}(v_{q-\vec{\eps}}))_{q\in I}$. In words, each section $(v_p)_{p\in I^\eps}$ of $M|_{I^\eps}$ is sent via $\tilde{\alpha}$ to the section $(\alpha_p(v_p))_{p+\vec{\eps}\in I}$ of $N_I$. 
Note that $(\alpha(v_p))_{p+\vec{\eps}\in I}$ is indeed a section of $N|_I$ by virtue of the naturality of $\alpha$ and the fact that for each $q\in I$, $q-\vec{\eps}\in I^\eps$. Now let us define $\tilde{\beta}$ as $\left[w_q\right]\mapsto \left[\beta_q(w_q)\right]$ for any $q\in I$ and $w_q\in N_q\hookrightarrow \bigoplus_{r\in I} N_r$. Again the naturality of $\beta:N\Rightarrow M(\vec{\eps})$ and the fact that for any $q\in I$, the point $q+\vec{\eps}$ belongs to $I^\eps$ ensure the well-definedness of $\tilde{\beta}$.

Fix any $p_0\in I$. Then, both  $p_0-\vec{\eps}$ and $p_0+\vec{\eps}$ belong to $I^\eps$. Now note that any section $(v_p)_{p\in I^\eps}\in \varprojlim M|_{I^\eps}$ is sent via the maps in diagram  (\ref{eq:commutative diagram}) as follows, thus establishing the desired commutativity.

\[\begin{tikzcd}
(v_p)_{p\in I^\eps} \arrow[r, mapsto] \arrow[d, mapsto, "\tilde{\alpha}"]
& \left[v_{p_0-\vec{\eps}}\right]=\left[\varphi_M(p_0-\vec{\eps},p_0+\vec{\eps})(v_{{p_0}-\vec{\eps}})\right]\\
(\alpha_p(v_p))_{p+\vec{\eps}\in I} \arrow[r, mapsto]
&  \left[\alpha_{p_0-\vec{\eps}}(v_{{p_0}-\vec{\eps}})\right] \arrow[u, mapsto, "\tilde{\beta}"] 
\end{tikzcd}
\]
\end{proof}